%% file: DeMarchisIanniPacella.tex
\pdfoutput=1
\documentclass[12pt,reqno,a4paper]{amsart}

\usepackage{graphicx}
\usepackage{color}
\usepackage{transparent}

\usepackage{mathrsfs}
\usepackage{amssymb,amsmath, amsfonts, amsthm}
\usepackage{latexsym}
\usepackage[dvips]{epsfig}

\newtheorem{theorem}{Theorem}[section]
\newtheorem{proposition}[theorem]{Proposition}
\newtheorem{lemma}[theorem]{Lemma}
\newtheorem{corollary}[theorem]{Corollary}
\newtheorem{remark}[theorem]{Remark}



\usepackage{epsfig,color,xcolor}
\newcommand{\rosso}[1]{\textcolor[rgb]{0.90,0.00,0.00}{#1}}

\newcommand{\ble}{\begin{lemma}}
\newcommand{\ele}{\end{lemma}}
\newcommand{\be}{\begin{equation*}}
\newcommand{\ee}{\end{equation*}}

\newcommand{\bel}{\begin{equation}}
\newcommand{\eel}{\end{equation}}

\DeclareMathOperator{\Id}{Id}

\newcommand{\al}{\alpha}

\newcommand{\ep}{\varepsilon}
\newcommand{\fr}{\frac }
\newcommand{\intopab}{\int_{\mathbb A_{p,\alpha}}}
\newcommand{\lap}{\Delta}
\newcommand{\ns}{4}
\newcommand{\N}{\mathbb{N}}
\newcommand{\na}{\nabla}
\newcommand{\Ompab}{\mathbb A_{p,\alpha}}

\newcommand{\p}{\fontsize{1}{1}{$+$}}
\newcommand{\sm}[1]{\fontsize{4}{4}{$#1$}}
\newcommand{\m}{\fontsize{1}{1}{$-$}}

\newcommand{\R}{\mathbb{R}}
\newcommand{\set}[1]{\left\{#1\right\}}
\newcommand{\arcup}[1]{\widehat{#1}}
\renewcommand{\to}{\rightarrow}
\newcommand{\To}{\longrightarrow}
\newcommand{\zpo}{\mathcal{Z}_{p,0}}

\def\sideremark#1{\ifvmode\leavevmode\fi\vadjust{\vbox to0pt{\vss
 \hbox to 0pt{\hskip\hsize\hskip1em
 \vbox{\hsize2.1cm\tiny\raggedright\pretolerance10000
  \noindent #1\hfill}\hss}\vbox to15pt{\vfil}\vss}}}%

\newcommand{\edz}[1]{\sideremark{#1}}
\begin{document}
\numberwithin{equation}{section}
\parindent=0pt
\hfuzz=2pt
\frenchspacing

\title[]{Sign changing solutions of Lane Emden problems with interior nodal line and semilinear heat equations}

\author[]{Francesca De Marchis, Isabella Ianni, Filomena Pacella}

\address{Francesca De Marchis, University of Rome {\em Tor Vergata}, Via della Ricerca
Scientifica 1, 00133 Rome, Italy}
\address{Isabella Ianni, Second University of Napoli, V.le Lincoln 5, 81100 Caserta, Italy}
\address{Filomena Pacella, University of Rome {\em Sapienza}, P.le Aldo Moro 8, 00185 Rome, Italy}



\thanks{Research partially supported by FIRB project: {\sl Analysis and Beyond} and PRIN 2009-WRJ3W7 grant. }

\begin{abstract}
We consider the semilinear Lane Emden problem
\begin{equation}\left\{\begin{array}{lr}-\Delta u= |u|^{p-1}u\qquad  \mbox{ in }\Omega\\
u=0\qquad\qquad\qquad\mbox{ on }\partial \Omega
\end{array}\tag{$\mathcal E_p$}\right.
\end{equation}
where $\Omega$ is a smooth bounded simply connected domain in $\mathbb R^2$, invariant by the action of a finite symmetry group $G$. \\
We show that if the orbit of each point in $\Omega$, under the action of the group $G$, has cardinality greater than or equal to $\ns$ then, for $p$ sufficiently large, there exists a sign changing solution of \eqref{problem} with two nodal regions  whose nodal line does not touch $\partial\Omega$.\\
 This result is proved as a consequence of an analogous result for the associated parabolic problem.
\end{abstract}
\maketitle

\section{Introduction}
We consider the semilinear elliptic problem

\begin{equation}\label{problem}\left\{\begin{array}{lr}-\Delta u= |u|^{p-1}u\qquad  \mbox{ in }\Omega\\
u=0\qquad\qquad\qquad\mbox{ on }\partial \Omega
\end{array}\right.\tag{$\mathcal E_p$}
\end{equation}

where $\Omega\subseteq \mathbb R^2$ is a smooth bounded domain and $p>1.$

\

In  this paper we address the question of the existence of sign changing solutions of \eqref{problem} with two nodal domains and whose nodal line does not touch $\partial\Omega.$
By this we mean that, denoting by
\bel\label{Zp}\mathcal Z_p=\{x\in\Omega, \ u_p(x)=0\}\eel
the nodal set of $u_p$, then $\overline{ \mathcal Z_p}\cap\partial\Omega=\emptyset.$

\

Obviously if $\Omega$ is a ball such a solution exists for any $p>1$, just considering the least energy nodal radial solution of \eqref{problem}. Moreover, by symmetry considerations, it has been proved in \cite{AftalionPacella} that the radial one is not the least energy nodal solution in the whole space $H^1_0(\Omega)$ and that it has Morse index at least three (actually at least four).\\
Thus it is natural to ask whether a sign changing solution with an interior nodal line exists for other domains.
\\Note that nodal solutions of \eqref{problem}, $p$ large, of different type, in particular with the nodal line intersecting the boundary have been found in \cite{EMP}.
\

By imposing some symmetry on $\Omega$ we are able to prove the following theorem, assuming, without loss of generality, that the origin $O\in\Omega$.

\begin{theorem}\label{thm}
Assume that $\Omega$ is simply connected and invariant under the action of a finite group $G$ of orthogonal transformations of $\R^2$. If  $|Gx|\geq\ns$ for any $x\in\bar\Omega\setminus\{O\}$,
then, for $p$ sufficiently large \eqref{problem} admits a sign changing $G$-symmetric solution $u_p$, with two nodal domains, whose nodal line neither touches $\partial\Omega$, nor passes through the origin.
\end{theorem}
In the above statement by $|Gx|$ we mean the cardinality of the orbit of the point $x$ under the action of the group $G$.

\begin{remark}\label{rmkGruppo}
Note that by Theorem 3.4 of \cite{Artin}  the above hypothesis on the group is equivalent to ask that $G=C_h$ or $G=D_h$ for some $h\geq \ns$, where
$C_h$ is the cyclic group of order $h$ generated by a rotation of $\frac{2\pi}{h}$ and  $D_h$ is the dihedral group of order $2h,$ generated by a rotation of $\frac{2\pi}{h}$ and a reflection about a line through the origin.
\end{remark}

The point of view taken in this article to prove Theorem \ref{thm} is to study an initial value problem for the associated semilinear heat equation and prove that, for $p$ large, it is possible to construct an initial datum for which the solution is global (in time), changes sign (at every time) and the corresponding $\omega$-limit set is nonempty and consists of solutions of \eqref{problem} having the desired properties.
\\
More precisely we consider the semilinear parabolic problem

\begin{equation}\label{ParabolicProblem}\left\{\begin{array}{ll}v_t-\Delta v= |v|^{p-1}v& \mbox{ in }\Omega\times (0,+\infty)\\
v=0&\mbox{ on }\partial \Omega\times [0,+\infty)\\
v(\cdot, 0)=v_0& \mbox{ in }\Omega
\end{array}\tag{$\mathcal P_p$}\right.
\end{equation}
and prove the following result.

\begin{theorem}\label{teoremaPrincipale}
Let the domain $\Omega$ and the group $G$ be as in Theorem \ref{thm}. Then there exists $p_0>1$ such that for any $p\geq p_0$ there exists a $G$-symmetric function  $v_0=v_{p,0}\in H^1_0(\Omega)$ having two nodal domains such that the corresponding solution $v_p(x, t)$ of \eqref{ParabolicProblem} is global, sign changing for every fixed $t,$ the $\omega$-limit set $\omega(v_0)$ is nonempty  and any function $u_p\in\omega(v_0)$ is a $G$-symmetric nodal solution of \eqref{problem} with two nodal domains and whose nodal line neither touches $\partial\Omega$, nor passes through the origin.
\end{theorem}

\

Obviously the result of Theorem \ref{thm} follows from Theorem \ref{teoremaPrincipale}.

\

The are several motivations to choose the parabolic approach. A first one is that, the study of \eqref{ParabolicProblem} and, in particular, the analysis of qualitative properties of the solutions of \eqref{ParabolicProblem} is interesting in itself. Another one is to show that the sign changing solution of \eqref{problem} that we get in Theorem \ref{thm} arises as a \emph{natural} evolution, by the heat flow, of a suitable ``initial'' function that we construct.
\\
Moreover we believe that the two problems \eqref{problem} and \eqref{ParabolicProblem}, which most of the time are analyzed in independent ways, have several common features.

\

One of the important steps in proving Theorem \ref{teoremaPrincipale} is the choice of the initial datum $v_{p,0}$. To do this we select a proper linear combination of  the positive radial solution of \eqref{problem} in the annulus
\bel\label{Apa}
\Ompab:=\{x\in\mathbb R^2 \ : \ e^{-\alpha p}<\ |x|\ < \ b\} \subset\Omega
\eel
and the negative radial solution of \eqref{problem} in the ball
\bel\label{Bpa}
\mathbb B_{p,\alpha}:=\{x\in\mathbb R^2 \ : \ |x|<e^{-\alpha p}\} \subset\Omega
\eel
for $\alpha>0$ and $b>0$ suitable chosen. We obtain in this way a function $v_{p,0}$ which is $G$-invariant, has two nodal regions and the nodal line does not intersect $\partial\Omega.$

Then to prove that the trajectory starting from $v_{p,0}$ ends up with a solution of \eqref{problem} having the same properties a good estimate from above of the energy of $v_{p,0}$, for $p$ large, is crucial. This is a delicate and nontrivial point of the proof (see Proposition \ref{lemmaStimeDatiIniziali}). It allows to prove, together with an estimate from below of the energy in each nodal region, and with the assumption on the symmetry group, namely $|G x|\geq \ns, \ \forall x\in\bar\Omega\setminus\{O\}$, that if the nodal line of the solution $u_p$ in $\omega(v_{p,0})$ touched the boundary or passed through the origin then too many nodal regions would be created and so the energy of $u_p$ would exceed that of $v_{p,0}$ which is not possible.
\\
For this last step as well as to show that the solution $v_p(x,t)$ of \eqref{ParabolicProblem}  does  change sign, $\forall t>0$,  some topological argument is needed.

\

We point out that a stronger result would be to show that along the trajectory, i.e. $\forall \ t\in (0,+\infty)$ the solution $v_p(x,t)$ has always two nodal regions and its nodal line does not touch the boundary or passes through the origin. We believe that this should be true (for energy reasons!) but we are not able to prove it at this stage.

\

We also remark that the same energy estimates for $v_{p.0}$ that we obtain here show that,  under the same assumption as in Theorem \ref{thm}, the least energy nodal solution in the subspace $H^G$ of $H^1_0(\Omega)$ of $G$-invariant functions has two nodal regions and the nodal line neither touches the boundary nor passes through the origin, if $p$ is sufficiently large.
We observe that for this minimality property some assumption on the action of the symmetry group $G$ is needed. Indeed in $2$-dimension the  least energy nodal solution of \eqref{problem} in the ball has a symmetry hyperplane (see \cite{BartschWethWillem, PacellaWeth}) but the nodal line touches the boundary (\cite{AftalionPacella}).

\

Another important issue is to have information about the Morse index $m(u_p)$ of the solution constructed in Theorem \ref{thm}. As mentioned before for least energy nodal radial solution in the ball the Morse index is at least three (\cite{AftalionPacella}). By using the same ideas as in \cite{AftalionPacella}, under an additional hypothesis on the symmetry group  we also get that $m(u_p)\geq3$.

\begin{theorem}\label{teoremaMorseIndex}
Under the same assumptions of Theorem \ref{thm}, if $G$ contains a reflection with respect to a line $T$ through the origin (namely $G=D_h$, for some $h\geq\ns$) and $\Omega$ is convex in the direction orthogonal to $T$, then $m(w_p)\geq3$, where $w_p$ is any $G$-symmetric solution of \eqref{problem} whose nodal line does not touch the boundary of $\Omega$.
\end{theorem}

\

Finally it would be interesting to study the asymptotic behavior as $p\to\infty$ of the solutions $u_p$ constructed in Theorem \ref{thm}. In view of the geometric properties of our solutions we believe that the behavior should be similar to that of the radial solutions in the ball studied in \cite{GrossiGrumiauPacella}. We plan to analyze this question in a future paper.

\

Further comments will be delayed to the specific sections.\\
The outline of the paper is as follows. In Section \ref{SectionEnergyEstimates} we define the functions used to construct the initial datum and prove the crucial energy estimates. In Section \ref{sectionParabolic} we recall some known properties of the semilinear heat flow and prove a preliminary result about sign changing solutions of \eqref{ParabolicProblem}. Finally in Section \ref{sectionProofs} we prove Theorem \ref{teoremaPrincipale} and Theorem \ref{teoremaMorseIndex}.

\section{Energy estimates}\label{SectionEnergyEstimates}
Let us consider in the Sobolev space $H^1_0(\Omega)$ the Nehari manifold
$$
\mathcal N_p:=\{u\in H^1_0(\Omega)\setminus\{O\}:\|\nabla u\|^2_2=\|u\|_{p+1}^{p+1}\}
$$
and the energy functional
$$
E_p(u):=\frac{1}{2}\|\nabla u\|^2_{2}-\frac{1}{p+1}\|u\|_{p+1}^{p+1}
$$
associated to problem \eqref{problem}, for $p>1$.

\

We start with a simple lemma.

\begin{lemma}\label{lemmaCombinazioneElementiNehari}
Let $u_1, u_2\in \mathcal N_p,$ $supp\  u_1\cap supp\ u_2=\emptyset.$
 Then $$E_p(t_1u_1+t_2u_2)\leq E_p(u_1)+E_p(u_2)\qquad \mbox{ for all }t_1, t_2\in\mathbb R.$$
\end{lemma}

\begin{proof}
\begin{eqnarray*}&&
E_p(t_1u_1+t_2u_2)=E_p(t_1u_1)+E_p(t_2u_2)=\sum_{i=1}^2 \left( \frac{t_i^2}{2}\|\nabla u_i\|^2_{2}-\frac{|t_i|^{p+1}}{p+1}\|u_i\|_{p+1}^{p+1}\right)\\
&=&
\sum_{i=1}^2 \left( \frac{t_i^2}{2}-\frac{|t_i|^{p+1}}{p+1}\right)\|\nabla u_i\|^2_{2}
\leq   \sum_{i=1}^2  \left( \frac{1}{2}-\frac{1}{p+1}\right)  \|\nabla u_i\|^2_{2}=E_p(u_1)+E_p(u_2).
\end{eqnarray*}
\end{proof}

\

Now let us denote by $B_b$ a ball centered at the origin with radius $b>0$ such that $B_b\subset \Omega$ and consider the annulus $\mathbb A_{p,\alpha}$ and the ball $\mathbb B_{p,\alpha}$ defined in \eqref{Apa} and \eqref{Bpa}.

The following proposition plays a crucial role in proving our results.

\begin{proposition}\label{lemmaStimeDatiIniziali}
For $\alpha >0$ let $u_{p,1,\alpha}$ be the unique positive radial solution  to \eqref{problem}  in $\Ompab$ and
$u_{p,2,\alpha}$ be the unique positive radial solution to \eqref{problem}  in $\mathbb B_{p,\alpha}$ .
Then $u_{p,1,\alpha}, u_{p,2,\alpha} \in \mathcal N_p$ and
for any $\epsilon>0$ there exists $p_{\epsilon}$ such that for $p\geq p_{\epsilon}$
$$ p E_p( u_{p,1,\alpha})\leq 4\pi e\frac{e^{2\alpha-1}}{\alpha}+\epsilon ,$$
$$p E_p( u_{p,2,\alpha}) \leq 4\pi e^{4\alpha+1}+\epsilon.$$
\end{proposition}

\begin{remark}The estimates given in Proposition \ref{lemmaStimeDatiIniziali} are accurate and give a sharp upper bound on the minimal energy of $G$-invariant sign-changing functions in $\mathcal N_p$, as needed to prove our results. Their proof is long and technically complicated and therefore we postpone it to the end of the section. However even a rough estimate of the minimal energy seems not easy to get.
\end{remark}

\begin{corollary}[Energy upper bound] \label{PropositionEnergyUpperBound}
$\,$\\
There exists $\bar\alpha >0$ such that for any $\epsilon >0$ there exists $p_{\epsilon}$ such that
$$pE_p(t_1u_{p,1,\bar\alpha}+t_2u_{p,2,\bar\alpha})\leq  4.97\ \cdot 4\pi e +\ \epsilon\qquad \mbox{ for all }t_1, t_2\in\mathbb R,\;\:\mbox{ for }p\geq p_{\epsilon}.$$
\end{corollary}

\begin{proof}
Let $\alpha>0$, since $u_{p,i,\alpha}\in\mathcal N_p$, $i=1,2$ and $supp \ u_{p,1,\alpha}\ \cap \ supp\ u_{p,2,\alpha}=\emptyset,$ by Lemma \ref{lemmaCombinazioneElementiNehari}
 $$pE_p(t_1u_{p,1,\alpha}+t_2u_{p,2,\alpha})\leq pE_p(u_{p,1,\alpha})+pE_p(u_{p,2,\alpha})\qquad \mbox{ for all }t_1, t_2\in\mathbb R.$$
We denote by $\bar\alpha$ the point of minimum in $(0,+\infty)$ of the scalar function $f(\alpha):=\frac{e^{2\alpha-1}}{\alpha}+e^{4\alpha }$.

Let $\epsilon>0,$ using Proposition \ref{lemmaStimeDatiIniziali} one has for $p\geq p_{\epsilon}$
$$ pE_p(u_{p,1,\bar\alpha})+pE_p(u_{p,2,\bar\alpha})=4\pi e  \min_{\alpha >0} f(\alpha)
+\epsilon\leq 4\pi e \,f\left(\tfrac{1}{5}\right) +\epsilon\leq 4.97\ \cdot 4\pi e + \epsilon.$$
\end{proof}

\

\begin{proposition}[Energy lower bound]\label{PropositionEnergyLowerBound}
Let $(u_p)_p$ be a family of nodal solutions of problem  \eqref{problem}. Then
for any $\epsilon>0$ there exists $p_{\epsilon}$ such that for $p\geq p_{\epsilon},$
$$pE_p(u_p\chi_{D_p})\geq  4\pi e -\epsilon$$
where $D_p\subset\Omega$ is any nodal domain of $u_p$ and $\chi_{D_p}$ is the characteristic function of $D_p$.
In particular if $u_p$ has  $k$ nodal domains then for any $\epsilon>0$
$$pE_p(u_p)\geq k 4\pi e -\epsilon,$$ for $p$ sufficiently large.
\end{proposition}
\begin{proof}
Since $u_p\chi_{D_p}\in\mathcal N_p$ one has that $$pE_{p}(u_p\chi_{D_p}) \geq  \inf_{\mathcal N_p} \left(pE_p\right) .$$
The conclusion follows passing to the limit and observing that
\begin{equation}
\label{limite}
\lim_{p\rightarrow\infty}\left(\inf_{\mathcal N_p} pE_p\right) =4\pi e
\end{equation}
(see \cite[Lemma 2.1]{AdimurthiGrossi} and also \cite{RenWei}).
%
%
\end{proof}

\

\

\

\begin{proof}[Proof of Proposition \ref{lemmaStimeDatiIniziali}]\label{sectionAppendix}
It is enough to prove that
\begin{equation} \label{prop:stimap}
p\intopab |\na u_{p,1,\alpha}|^2\,dx \leq \fr{8\pi e^{2\al}}{\al}+o_p(1), \qquad\textrm{as $p\to+\infty$,}
\end{equation}
\begin{equation} \label{prop:stimaPalla}
p\int_{\mathbb B_{p, \alpha}} |\na u_{p,2,\alpha}|^2\,dx =  8\pi e \cdot e^{4\alpha} + o_p(1), \qquad\textrm{as $p\to+\infty$.}
\end{equation}
Indeed, since $u_{p,i,\alpha}\in\mathcal N_p,$ $i=1,2,$ one has that
\begin{eqnarray*}pE_p(u_{p,i,\alpha})=
p\left( \frac{1}{2}-\frac{1}{p+1}\right) \|\nabla u_{p,i,\alpha}\|^2_2 \leq  \frac{1}{2} p\|\nabla u_{p,i,\alpha}\|^2_2, \mbox{ for any }p>1.
\end{eqnarray*}

\

\

\emph{STEP 1: } we prove \eqref{prop:stimap}.

It is easy to see that $u_{p,1,\alpha}=\alpha_p^{\frac{1}{p-1}}z_p,$ where $z_p$ is a minimizer
for
$$
I_p:=\inf_{
\begin{array}{cr}u\in H^1_{0,r}( \Ompab)\\ u\neq 0\end{array}
}
\frac{\int_{ \Ompab}|\nabla u|^2}{
\left(\int_{ \Ompab}u^{p+1}\right)^{\frac{2}{p+1}}
}
$$
and $\alpha_p=\int_{ \Ompab}|\nabla z_p|^2/\int_{ \Ompab}z_p^{p+1}.$
Hence one has
\begin{eqnarray*}p\int_{ \Ompab}|\nabla u_{p,1,\alpha}|^2dx &=& p
\alpha_p^{\frac{2}{p-1}}\int_{ \Ompab}|\nabla z_p|^2dx=
p\left[
\frac{\int_{ \Ompab}|\nabla z_p|^2}{\left(\int_{ \Ompab}z_p^{p+1}\right)^{\frac{2}{p+1}}}\right]^{\frac{p+1}{p-1}}
\\
&
\leq &
p \left[
\frac{\int_{ \Ompab}|\nabla \omega_p|^2}{\left(\int_{ \Ompab}\omega_p^{p+1}\right)^{\frac{2}{p+1}}}\right]^{\frac{p+1}{p-1}},
\end{eqnarray*}
where $\omega_p$, introduced in \cite{Grossi}, is defined as follows:

$$\omega_p(x)=\omega_p(|x|)= \frac{2}{\al p+\log b }
\left\{
\begin{array}{ll}
\al p +\log |x|& e^{-\al p}\leq |x|\leq b^{\fr12}e^{-\fr12 \al p}\\
\log{\frac{b}{|x|} }& b^{\fr12}e^{-\fr12 \al p}\leq |x|\leq b
\end{array}\right.
$$

Since

$$\int_{\Ompab}|\nabla \omega_p|^2dx=\frac{8\pi}{(\al p+\log b)^2}\int_{e^{-\al p}}^b\frac{1}{r}dr=\frac{8\pi}{\al p+\log b},$$
we get

\bel\label{Mpab}
p\intopab |\na u_{p,1,\alpha}|^2\,dx \leq p\,\frac{(8\pi)^{\frac{p+1}{p-1}}}{(\al p+\log b)^{\fr{p+1}{p-1}}}\fr{1}{(\intopab \omega_p^{p+1})^{\fr 2{p-1}}}.
\eel
To estimate $p\intopab |\na u_{p,1,\alpha}|^2\,dx$ one needs to control
\begin{eqnarray}\label{help}
\int_{\Ompab}\omega_p^{p+1}&=&\frac{2^{p+2}\,\pi}{(\al p+\log b)^{p+1}}\left[{\int_{e^{-\al p}}^{b^{\fr12}e^{-\fr12\al p}}(\al p+\log r)^{p+1}\,r\,dr} + {\int_{b^{\fr12}e^{-\fr12\al p}}^b \log^{p+1}(\fr b r)\,r\,dr} \right]\nonumber\\
&\geq & \frac{2^{p+2}\,\pi}{(\al p+\log b)^{p+1}}{\int_{e^{-\al p}}^{b^{\fr12}e^{-\fr12\al p}}(\al p+\log r)^{p+1}\,r\,dr} \nonumber\\
&\stackrel{s=e^{\fr12\alpha p}r}{=}& \frac{2^{p+2}\,\pi\, e^{-\al p}}{(\al p+\log b)^{p+1}} \int_{e^{-\fr12\al p}}^{b^{\fr12}}\left(\fr{\al p}2+\log s\right)^{p+1}\,s\,ds \nonumber\\
&=& \frac{2^{p+2}\,\pi\, e^{-\al p}}{(\al p+\log b)^{p+1}}\: \left(\fr{\al p}{2}\right)^{p+1} \int_{e^{-\fr12\al p}}^{b^{\fr12}}\left(1+\fr{\fr2\al\log s}{p}\right)^{p+1}\,s\,ds \\
&=& \frac{\,2\,\pi\, e^{-\al p}\,({\al p})^{p+1}}{(\al p+\log b)^{p+1}}  \int_{0}^{b^{\fr12}}s^{\fr2\al+1}\,ds+o_p(1) \quad \mbox{as }p\to +\infty. \nonumber
\end{eqnarray}
Thus setting $c_{\al,b}=(\int_{0}^{b^{\fr12}}s^{\fr2\al+1}\,ds)$ and substituting \eqref{help} into \eqref{Mpab} one obtains the desired estimate
\begin{eqnarray*}
p\intopab |\na u_{p,1,\alpha}|^2\,dx &\leq& p\,\frac{(8\pi)^{\frac{p+1}{p-1}}}{(\al p+\log b)^{\fr{p+1}{p-1}}}\fr{(\al p+\log b)^{2\fr{p+1}{p-1}}}{\left(2\pi e^{-\al p}(\al p)^{p+1} c_{\al,b}\right)^\fr2{p-1}}\\
&=&\fr{8\pi\,e^{2\al}}{\al}+o_p(1),\qquad\textnormal{as $p\to+\infty$.}
\end{eqnarray*}

\

\emph{STEP 2: } proof of \eqref{prop:stimaPalla}.

Since $u_{p,2}$ is radial we define $u_{p,2}(r):=u_{p,2}(|x|),$ and  $u_{p,2}(r)=e^{\frac{2\alpha p}{p-1}}w_p(r e^{\alpha p})$ where
$w_p$ is the unique positive solution of $-\Delta u=u^p$ in the unit ball $B_{1}(0).$\\
It is easy to see that $w_{p}=\alpha_p^{\frac{1}{p-1}}z_p,$ where $z_p$ is a minimizer
for
$$
I_p:=\inf_{
\begin{array}{cr}u\in H^1_{0}( B_1(0))\\ u\neq 0\end{array}
}
\frac{\int_{ B_1(0)}|\nabla u|^2}{
\left(\int_{ B_1(0)}u^{p+1}\right)^{\frac{2}{p+1}}
}
$$
and $\alpha_p=\int_{B_1(0)}|\nabla z_p|^2/\int_{ B_1(0)}z_p^{p+1}.$
Hence from \cite[Lemma 2.1]{AdimurthiGrossi} (see also \cite{RenWei}) it follows that
\begin{eqnarray*}p\int_{ B_1(0)}|\nabla w_{p}|^2dx &=& p
\alpha_p^{\frac{2}{p-1}}\int_{ B_1(0)}|\nabla z_p|^2dx=
p\left[
\frac{\int_{ B_1(0)}|\nabla z_p|^2}{\left(\int_{ B_1(0)}z_p^{p+1}\right)^{\frac{2}{p+1}}}\right]^{\frac{p+1}{p-1}}
=p I_p^{\frac{p+1}{p-1}}\stackrel{p\to+\infty}{\To} 8\pi e,
\end{eqnarray*}
namely $$p\int_0^1 w_p'(r) ^2r dr\stackrel{p\to+\infty}{\To} 4 e.$$
As a consequence
\begin{eqnarray*}p\int_{\mathbb B_{p, \alpha}} |\na u_{p,2}|^2\,dx = 2\pi p\int_0^{e^{-\alpha p}}u_{p,2}'(r)^2rdr= 2\pi p e^{\frac{4\alpha p}{p-1}}\int_0^1 w_p'(r)^2rdr \stackrel{p\to+\infty}{\To} 8\pi  e^{4\alpha+1}. \end{eqnarray*}
\end{proof}

\

\section{A preliminary result on the parabolic problem}\label{sectionParabolic}

We start by recalling some well known facts about the parabolic problem \eqref{ParabolicProblem}. This problem  has been extensively studied in the last years by many authors, we refer to the  monograph \cite{QuittnerSoupletBook} for more results  and  further references.

\

Here $\Omega$ is any smooth bounded domain in $\mathbb R^2$ (but similar results hold also in $\mathbb R^n, \ n\geq 2$ with obvious changes).

\

Let $X:=W_0^{1,q}(\Omega)$ with $q>2$ and $Y:=\{u\in C^1(\overline{\Omega}): u|_{\partial\Omega}=0\},$ then
$Y\hookrightarrow X\hookrightarrow \{u\in C^0(\overline{\Omega}): u|_{\partial\Omega}=0\}$ (where the second inclusion is just the Sobolev embedding since $q>2$).

\

We first consider the local solvability of \eqref{ParabolicProblem} (see \cite[Appendix E]{QuittnerSoupletBook}):
\begin{proposition}\label{prop:localex}
For every $v_0\in X$ the IBVP \eqref{ParabolicProblem} has a unique
solution $v(t)=\varphi(t,v_0)\in C([0,T),X)$ with maximal existence time $T:=T(v_0)>0$ which is a classical solution for $t\in (0,T).$

The set $\mathcal G:=\{(t,v_0): t\in [0,T(v_0)),\  v_0\in X\}$ is open in $[0,\infty)\times X,$ and $\varphi:\mathcal G\rightarrow X$ is a semiflow on $X.$
\end{proposition}

\

Moreover the following continuity property with respect to the initial datum in stronger norm holds (see for instance  \cite[Appendix E]{QuittnerSoupletBook}).

\begin{proposition}\label{continuitaDatoIn}
For every $v_0\in X$ and every $t\in (0,T(v_0))$ there is a neighborhood $U\subset X$ of $v_0$ in $X$ such that $T(v)>t$ for $v\in U,$ and $\varphi (t,\cdot):(U,\|\cdot\|_X)\rightarrow (Y,\|\cdot\|_{Y})$ is a continuous map.
\end{proposition}

\

In the sequel we will often write $\varphi^t(v)$ instead of $\varphi(t,v).$\\
Observe that the nonlinearity in \eqref{ParabolicProblem} is odd, hence by uniqueness it follows that the map $X\ni v\mapsto \varphi^t(v)$ is odd.\\

For a classical solution $v$ of \eqref{ParabolicProblem}, it is easy to show that
$\frac{d}{dt}E_p(v)
=-\|v_t\|_2^2
$,
hence $E_p$ is strictly decreasing along nonconstant trajectories $t\mapsto \varphi(t,v_0)$ in $X$, namely the energy is a strict Lyapunov functional.

\

As a consequence, since  $0$ is a strict local minimum for $E_p$,  it follows that the constant solution $v\equiv 0$ is asymptotically stable in $X.$ Let  $\mathcal A_p^{\star}$ be its domain of attraction, i.e. :
$$ \mathcal A_p^{\star}:= \{v\in X:T(v)=+\infty \mbox{ and }\varphi(t,v)\rightarrow 0\mbox{ in }X\mbox{ as }t\rightarrow \infty\}.$$ The asymptotic stability of $0$, the semiflow properties of solutions of \eqref{ParabolicProblem} and  the continuous dependence of solutions on initial data imply that the set $\mathcal A_p^{\star}$ is an open neighborhood of $0$ in $X.$

\

Let $\partial \mathcal A_p^{\star}$ denote the  boundary of the set $\mathcal A_p^{\star}$ in $X$ with respect to the $X$-topology.

Since $\mathcal A_p^{\star}$ is open and $0$ is asymptotically stable, the continuous dependence of the semiflow $\varphi$ on the initial values implies that $\partial\mathcal A_p^{\star}$ is positively invariant under $\varphi.$  Moreover it is invariant with respect to the antipodal symmetry $x\mapsto -x$ (since $v\mapsto\varphi^t(v)$ is odd).

\

We also recall  the following global existence result  (see \cite[Appendix G]{QuittnerSoupletBook} for the definition and properties of the $\omega$-limit set and also \cite{GazzolaWeth} for a similar result)
\begin{proposition}Let $v_0\in\partial\mathcal A_p^{\star}.$ Then
\begin{itemize}

\item[i.] $T(v_0)=\infty$ and for every $\delta>0$ the set $\{\varphi^t(v_0)\ : \ t\geq\delta\}\ (\subset \partial\mathcal A_p^{\star})$ is relatively compact in $Y$

\item[ii.]
the $\omega$-limit set $$\omega(v_0):=\bigcap_{t>0}clos_{Y}\left(\{\varphi(s,v_0):s\geq t\}\right)\ \subset \partial\mathcal A_p^{\star}$$
is a nonempty compact subset of $Y$ consisting of solutions of \eqref{problem}.
\end{itemize}
\end{proposition}

\

Now, using a topological argument based on the Krasnoselskii genus we prove the existence of a nodal solution for the parabolic problem \eqref{ParabolicProblem} (see Theorem \ref{theoremSolution} below).
This preliminary result will be used in Section \ref{sectionProofs} to prove Theorem \ref{teoremaPrincipale}.

This approach is quite similar to the one first introduced in \cite{WeiWeth} (for systems of two coupled equations and then used in \cite{Ianni} for a nonlocal scalar equation) and allows to select a special initial value on $\partial\mathcal A^{\star}$ for which the corresponding solution has the desired properties. Observe that here, unlike \cite{WeiWeth, Ianni} we are not working in a radial setting. In particular we cannot use the "zero number property" (which is satisfied by radial solutions, see \cite[Appendix F]{QuittnerSoupletBook}) and which is at the core of the Wei-Weth topological approach, but we use only the maximum principle for parabolic equations. For this reason we cannot obtain at this stage any additional information on the number of nodal domains along the flow (unlike \cite{WeiWeth, Ianni} where the exact number of nodal domains is established).\\
Indeed to get the result about the number of nodal regions of the functions in the $\omega$-limit set we will use, in the next section, the action of the group $G$ and the energy estimates of the previous section.

\

For a closed  subset $B\subset\partial\mathcal A_p^{\star},$ invariant with respect to the antipodal symmetry,  we denote by $\gamma (B)$ the usual Krasnoselskii genus and we recall some of the properties we will need:
\begin{lemma}\label{lemmaGenus} Let $A, B\subset\partial\mathcal A_p^{\star}$ be closed and invariant with respect to the antipodal symmetry.
\begin{itemize}
\item[(i)] If $A\subset B,$ then $\gamma(A)\leq \gamma(B).$
\item[(ii)] If $h:A\rightarrow\partial\mathcal A_p^{\star} $ is continuous and odd, then $\gamma(A)\leq\gamma(\overline{h(A)}).$
\item[(iii)] If $S$ is an invariant with respect to the antipodal symmetry, bounded neighborhood of the origin in a $k$-dimensional normed vector space and $v:\partial S\rightarrow \partial \mathcal A_p^{\star}$ is continuous and odd,  then $\gamma(v(\partial S))\geq k. $
\end{itemize}
\end{lemma}

\

Let $u_1, u_2\in X,$ $supp\  u_1\cap supp\ u_2=\emptyset,$
and consider the $2$-dimensional subspace $W_p(u_1, u_2)\subset X$ spanned by the functions $u_{1}$ and $u_{2}.$ Let \begin{equation}\label{O_p}\mathcal O_p(u_1,u_2):=W_p(u_1,u_2)\cap \mathcal A_p^{\star}.\end{equation}
\begin{proposition}\label{bordoO}
 $\partial\mathcal O_p(u_1,u_2)\subset\partial\mathcal A_p^{\star}$ is (symmetric and) compact and $\gamma (\partial\mathcal O_p(u_1,u_2))\geq 2.$
\end{proposition}
\begin{proof}
Since $E_p(s u_{i})\rightarrow -\infty$ as $|s|\rightarrow +\infty,$ $i=1,2,$ it is easy to see that $$\lim_{{\scriptsize{\begin{array}{lr}\|w\|_X\rightarrow +\infty\\ w\in W_p(u_1,u_2)\end{array}}}}E_p(w)=-\infty.$$  We also know that
$$E_p(w)\geq 0, \qquad\forall w\in \mathcal A_p^{\star},$$
hence $\mathcal O_p(u_1,u_2)$ is a symmetric bounded open neighborhood of $0$ in $W_p(u_1,u_2).$
As a consequence  $\partial\mathcal O_p(u_1,u_2)\subset\partial\mathcal A_p^{\star}$ is compact and, by the property (iii) in Lemma \ref{lemmaGenus}, $\gamma (\partial\mathcal O_p(u_1,u_2))\geq 2.$
\end{proof}

\

We now define
\begin{equation}\label{A^1}\mathcal A_p^1:=\{u\in\partial\mathcal A_p^{\star}\ :\ u\geq 0 \mbox{ or } u\leq 0\}\end{equation}

\begin{lemma}\label{lemmaGenusA1}
$\mathcal A_p^1$ is a closed subset of $X$, is positively invariant  and
$\gamma (\mathcal A_p^1)\leq 1.$
\end{lemma}
\begin{proof}
The closure in $X$ is trivial: since $(u_n)_n\ \subset \mathcal A_p^1,$ $u_n\rightarrow_n v$ in $X$ implies, by Sobolev embedding, that $u_n\rightarrow v$ uniformly, hence $v$ doesn't change sign, moreover $\partial \mathcal A_p^{\star}$ is closed in $X$.

The positive invariance follows from the positive invariance of $\partial\mathcal A_p^{\star}$ and from the maximum principle for the parabolic equation (see for instance \cite[Appendix F]{QuittnerSoupletBook}).

Finally we show that $\gamma (\mathcal A_p^1)\leq 1.$
Since $0\not\in \partial\mathcal A_p^{\star}$ we have
$\mathcal A_p^1= B_+\cup B_-$ with disjoint subsets $B_{\pm}$ defined by
$$B_+:=\{u\in\partial\mathcal A_p^{\star}\: \ u\geq 0, u\not\equiv 0\}$$
$$B_-:=\{u\in\partial\mathcal A_p^{\star}\: \ u\leq 0, u\not\equiv 0\}.$$
$B_{\pm}$ are relatively open in $\mathcal A_p^1$ hence the map
$$h:\mathcal A_p^1 \rightarrow \mathbb R\setminus \{0\}, \qquad h(u):=\left\{
\begin{array}{rr}
1\qquad u\in B_+\\
-1 \qquad u\in B_-\end{array} \right.$$
is continuous, moreover it is also odd, and this concludes the proof.
\end{proof}
%
%
\

We define also the closed  subsets of $\partial \mathcal A_p^{\star}$
$$\mathcal C_p^{1,t}:=\{v\in\partial\mathcal A_p^{\star}:\varphi^t(v)\in \mathcal A_{p}^1\}\quad\mbox{ for }t>0.$$
\begin{proposition}\label{propositionC}
 $\mathcal A_{p}^1\subset \mathcal C_p^{1,t}$ and
$\gamma(\mathcal C_p^{1,t})=\gamma (\mathcal A_{p}^1)\leq 1$ for every $t>0.$
\end{proposition}
\begin{proof} The inclusion is a consequence of the positive invariance of $\mathcal A_p^{1}$. Property (ii) in Lemma \ref{lemmaGenus} implies that $\gamma (\mathcal C_p^{1,t})\leq 1.$
Indeed $\gamma (\mathcal C_p^{1,t})\leq \gamma (\overline{\varphi^t(\mathcal C_p^{1,t})})$ since the map $\varphi^t:\mathcal C_p^{1,t}\rightarrow \partial \mathcal A_p^{\star}$ is continuous and odd, and $\gamma (\overline{\varphi^t(\mathcal C_p^{1,t})})\leq 1$ because of the inclusion $\overline{\varphi^t(\mathcal C_p^{1,t})}\subset \mathcal A_{p}^1,$ using   (i) in Lemma  \ref{lemmaGenus} and Lemma \ref{lemmaGenusA1}.
\end{proof}

\

In order to prove the main result of this section we need to introduce the following set
 \bel\label{Y1}Y_1:=\{u\in Y\ : \ u\geq 0 \mbox{ or } u\leq 0\}\eel

\begin{lemma} \label{lemmaSoluzioniInterne}If $u\in \mathcal A_p^1$ is a solution of \eqref{problem}, then $u\in int_Y(Y_1)$.
\end{lemma}
\begin{proof} Assume without loss of generality that  $u>0$ in $\Omega$.
If by contradiction $u\in\partial Y_1$, then there exists a sequence $(u_n)\subset Y\setminus Y_1=Y\setminus\bar Y_1$ ($Y_1$ is closed in $Y$) such that $u_n\stackrel{n\to+\infty}{\To} u$ in $Y$, namely $u_n\stackrel{n\to+\infty}{\To} u$ in $L^\infty(\Omega)$ and $\tfrac{\partial u_n}{\partial x_i}\stackrel{n\to+\infty}{\To} \tfrac{\partial u}{\partial x_i}$ in $L^\infty(\Omega)$, $i=1,2$. Since for any $n\in\N$, $u_n$ changes sign, then there exists $x_n\in\Omega$ such that $u_n(x_n)<0$. Up to a subsequence $x_n\stackrel{n\to+\infty}{\To} \bar x\in \bar\Omega$ and $u(\bar x)=0$, thus $\bar x\in \partial \Omega$.\\
Denoting by $y_n$ the projection of $x_n$ on $\partial\Omega,$
 by Lagrange Theorem we have $u_n(y_n)-u_n(x_n)=\langle\nabla u_n(\xi_n),y_n-x_n\rangle>0$, with $\xi_n=t_ny_n+(1-t_n)x_n$ for some $t_n\in(0,1).$ Thus $\xi_n\stackrel{n\to+\infty}{\To}\bar x$ and $\frac{\partial u}{\partial\nu}(\bar x)\geq 0$ where $\nu$ is the outer normal to $\partial\Omega$ in $\bar x$,  against the  Hopf Lemma.
\end{proof}

Similarly as in \cite{WeiWeth} and \cite{Ianni} we use the previous lemma together  with the continuity property w.r.t. initial data of the solutions of \eqref{ParabolicProblem} (Proposition \ref{continuitaDatoIn}) to prove the following
\begin{proposition}\label{propositionInterno}
Let $v_0\in\partial\mathcal A^{\star}_p$ such that $\omega(v_0)\cap \mathcal A_p^1\neq \emptyset$ and let $(v_n)\subset\partial\mathcal A^{\star}_p$ be a sequence such that $v_n\rightarrow_n v_0$ in $X$.  Then there exist $\bar t>0$ and $\bar n\in\mathbb N$ such that $$\varphi^t(v_n)\in \mathcal A_p^1\  \mbox{ for all }\  t\geq\bar t, n\geq\bar n.$$
\end{proposition}
\begin{proof}
Since
$\omega( v_{0})$ consists of solutions of \eqref{problem},  Lemma \ref{lemmaSoluzioniInterne}  implies that $\omega(v_{0})\cap int_{Y}(Y_{1})\neq\emptyset$, where $Y_1$ is the set defined in \eqref{Y1}. By the definition of the $\omega$-limit set it follows that there exists $\bar t>0$ such that $\varphi^{\bar t}(v_0)\in int_{Y}(Y_{1}) $. By Proposition \ref{continuitaDatoIn} there exists $\bar n\in\mathbb N$ such that
$\varphi^{\bar t}(v_n)\in int_{Y}(Y_{1}) $ for every $n\geq\bar n$. Since $\partial\mathcal A^{\star}_p$ is positively invariant, in particular $\varphi^{\bar t}(v_n)\in \mathcal A_p^1 $ for every $n\geq\bar n$. The conclusion follows from the positive invariance of  $\mathcal A_p^1$ (Lemma \ref{lemmaGenusA1}).
\end{proof}
\begin{remark}
We underline that the result in Lemma \ref{lemmaSoluzioniInterne} in general is not true if one substitutes the $Y$-topology with the $X$-topology. Fortunately the continuity property in Proposition \ref{continuitaDatoIn} holds in the $Y$-norm.
\end{remark}

\

 Let $\partial \mathcal O_p(u_1,u_2)$ and $\mathcal A_{p}^1$ the ones  defined respectively in \eqref{O_p} and \eqref{A^1}, we can now prove the main result of this section:

\begin{theorem}\label{theoremSolution} $\forall p>1$
there exists an initial condition $ v_{p,0}\in \partial \mathcal O_p(u_1,u_2)\setminus \mathcal A_{p}^1$ such that the (global) solution of \eqref{ParabolicProblem} $\varphi_p^t(v_{p,0})\in \partial\mathcal A_p^{\star}\setminus \mathcal A_p^1$, $\forall  t\in (0,+\infty)$.  Moreover $\emptyset\neq\omega(v_{p,0})\subset\partial\mathcal A_p^{\star}\setminus \mathcal A_p^1.$
\end{theorem}
\begin{proof}  The proof consists in constructing a suitable initial condition $v_{p,0}$ in $\partial \mathcal O_p(u_1,u_2)\setminus \mathcal A_{p}^1$ as the limit of a sequence of initial conditions $(v_n)$ suitably chosen (using a genus argument).
\\
More precisely, since by Propositions \ref{bordoO} $\gamma (\partial\mathcal O_p(u_1,u_2))\geq 2$ and by Proposition \ref{propositionC} $\gamma(\mathcal C_p^{1,t})\leq1$ for every $t>0$, we deduce that $\emptyset\neq\partial \mathcal O_p(u_1,u_2)\setminus \mathcal C_p^{1,t}$ for every $t>0$.
In particular for any sequence $t_n\rightarrow +\infty$  there exists $v_n\in \partial \mathcal O_p(u_1,u_2)\setminus \mathcal C_p^{1,t_n} (\subset \partial \mathcal A^{\star}_p\setminus\mathcal A^1_p)$ and,
since $\partial\mathcal O_p(u_1,u_2)$ is compact, we may pass to a subsequence such that $v_n\rightarrow  v_{p,0}\in\partial\mathcal O_p(u_1,u_2)\subset \partial\mathcal A_p^{\star}$ as $n\rightarrow \infty.$
Obviously $\omega(v_{p,0})\subset \partial\mathcal A_p^{\star}.$ Similarly as in \cite{WeiWeth, Ianni} we now prove that $\omega( v_{p,0})\subset\partial \mathcal A_p^{\star}\setminus\mathcal A_{p}^1.$
Assume  by contradiction that  $\omega( v_{p,0})\cap \mathcal A_{p}^1\neq \emptyset$, than
by Proposition \ref{propositionInterno} there exist $\bar t>0$ and $\bar n\in\mathbb N$ such that
$\varphi^t(v_n)\in \mathcal A_p^1$ for every $t\geq\bar t, n\geq\bar n$, hence $\varphi^{t_n}(v_n)\in \mathcal A_p^1$ for $n$ sufficiently large, reaching a contradiction. Indeed by construction $v_n\not\in C_p^{1,t_n}, $ namely
$\varphi^{t_n}( v_n)\not \in \mathcal A_p^1$.
%
%
%
%
%
%
%
%
%

Last by the positive invariance of $\mathcal A_p^1$ we also have that $\varphi_p^t(v_{p,0})\in \partial\mathcal A_p^{\star}\setminus \mathcal A_p^1$, $\forall t\in (0,+\infty)$.
\end{proof}

\section{Proof of Theorem \ref{teoremaPrincipale}}\label{sectionProofs}

The proof of Theorem \ref{teoremaPrincipale}  relies on three main ingredients: the \emph{preliminary results} in Section \ref{sectionParabolic} (in particular Theorem \ref{theoremSolution}) (to obtain nodal solutions), the \emph{energy estimates} in Section \ref{SectionEnergyEstimates} (to avoid more than two nodal domains in the $\omega$-limit)
and a geometrical argument in the presence of symmetry (to avoid that in the $\omega$-limit the nodal line could touch the boundary or contain the origin).
\edz{\rosso{*}}

\

Before getting started with the proof, we show general lemmas relating the $G$-invariance of $\Omega$ with the properties of the nodal line and the nodal domains of solutions of problem \eqref{problem} in $\Omega$. In order to clarify and state the results in a rigorous way, let us introduce some notations.
\edz{\rosso{*}}

\

Let $G$ be a finite subgroup of the orthogonal group $O(2)$ on $\R^2$, such that $|Gx|\geq\ns$ for any $x\in\bar \Omega\setminus\{O\}$. By Remark \ref{rmkGruppo} $G$ contains the cyclic group of rotations $C_h$ (for some $h\geq\ns$); in the following we will denote by $g$ the rotation of $2\pi/h$, which is a generator of $C_h$, hence $$C_h=\{g^0, g^1,\ldots,g^{h-1}\}\quad\textrm{where $g^0=\Id$, $g^k=\underbrace{g\circ\ldots\circ g}_{k}$ for $k=1,\ldots,h-1$.}$$
A function $u\in H^1_0(\Omega)$ is said to be $G$-symmetric if
$$
\textnormal{$u(\gamma x)=u(x)$ for any $\gamma\in G$ and a.e. $x$ in $\Omega$}.
$$
Moreover given a solution $u_{p}$ of \eqref{problem} we recall that
we denote by $\mathcal{Z}_{p}$ the nodal set of $u_{p}$, namely $$\mathcal{Z}_p=\set{x\in\Omega\,|\,u_p(x)=0}.$$

Finally in the sequel we may call in short curve the image of a curve.

\ble\label{lemmaA}
Let $u_p$ be a $G$-symmetric sign-changing solution to \eqref{problem} with at most four nodal regions. Then its nodal line does not passes through the origin $O$.
\ele

\begin{proof}
We recall first that if a point $x_0\in\Omega$ belongs to the nodal set, then there exists a positive radius $R$ such that $\{u_p^{-1}(0)\}\cap B(x_0,R)$ is made of $2n$ $C^1$-simple arcs, for some integer $n$, which all end in $x_0$ and whose tangent
lines at $x_0$ divide the disc into $2n$ angles of equal amplitude (see \cite{HW} or Theorem 2.1 of \cite{HHT}).

\

Assume by contradiction that $O \in \mathcal{Z}_p$.  By the properties of the nodal set just recalled we have that there exists a ball $B_r:=\{x\in\R^2,\,|x|<r\}$ such that there exists $x\in\mathcal{Z}_p\cap\partial B_r$ and a regular curve $\gamma_x\subset \mathcal{Z}_p\cap \overline{B_r}$ joining $O$ to $x$ having the following properties:
\begin{itemize}
\item[1.] $\gamma_x\setminus\{x\}\subset B_r, $
\item[2.] $\gamma_x\cap g^i(\gamma_x)=\{O\}$, for any $i=1, \ldots ,h-1$
\item[3.] $\gamma_x\subset\partial D_0^+$, where $D_0^+$ is a nodal domain  where $u_p>0$ (and so, by the strong maximum principle, also $\gamma_x\subset\partial D_0^-$, where $D_0^-$ is a nodal domain  where $u_p<0$)
\end{itemize}
Because of the symmetry, for any $i\in\{1,\ldots,h-1\}$ the curve $ \gamma_{g^ix}:=g^i(\gamma_x)\subset \mathcal{Z}_p\cap \overline{B_r}$ joins $O$ to $g^ix$ and satisfies the analogous properties:
\begin{itemize}
\item[1.] $\gamma_{g^ix}\setminus\{g^ix\}\subset B_r, $
\item[2.] $\gamma_{g^ix}\cap \gamma_{g^jx}=\{O\}$, for any  $i\neq j$
\item[3.] $\gamma_{g^ix}\subset\partial D_i^+$, where $D_i^+:= g^i(D_0^+)$ and $\gamma_{g^ix}\subset\partial (D_i^-)$ where $D_i^-:= g^i(D_0^-)$.
\end{itemize}
Clearly by construction the sets $(D^{\pm}_i\cap B_r)$ are pairwise disjoint.

\

Moreover, let us fix two points $d_0^\pm$ in $D^{\pm}_0\cap \partial B_r$; by symmetry for any $i\in\{1,\ldots,h-1\}$ $d_i^\pm:=g^i(d_0^\pm)\in D^{\pm}_i\cap \partial B_r$.
Since $D^\pm_0$ are connected, there exist two piecewise regular curves $\gamma_0^\pm:[0,1]\to (D^\pm_0\cup \{0\})$ joining the points $O$ and $d_0^\pm$, more precisely such that $\gamma_0^\pm(0)=O$ and $\gamma_0^\pm(t)\in D^\pm_0$ for any $t\in(0,1]$. For any $i\in\{1,\ldots,h-1\}$, we set $\gamma_i^{\pm}(t):=g^i(\gamma^\pm_0(t))$.

\

If the nodal domains  $D^+_k$, $k=0, \ldots, h-1$ were pairwise disjoint, then $u_p$ would have at least $h+1\geq 5$ nodal domains (at least $h$ where it is positive and at least $1$ where it is negative), which leads to a contradiction.\\

Hence we can assume that $D^+_k=D^+_j$ for some $k\neq j$. \\

First we show that in this case $D^+_0=D^+_1=\ldots = D^+_{h-1}.$

This is a direct consequence of the symmetry in the case  $j=k+1$ (or $j=0$ when $k=h-1$), indeed for every $i=0,\ldots ,h-1$ one has
$$D^+_i=g^{i-k}(D^+_k)=g^{i-k}(D^+_j)=g^{i-k+j}(D^+_0)=D^+_{i+1}.$$
When $j\neq k+1$ (or $j\neq 0$ in the case $k=h-1$) we need to exploit more the properties of $\mathbb R^2$. Let us assume without loss of generality that $k=0<j$.
Being $D^+_0=D^+_j$ connected, there exists a piecewise regular curve $\varphi^+$ connecting $d^+_0$ and $d^+_j$ in $D^+_0$.\\
By definition of $\gamma^+_0$ and since $O\notin \varphi^+$, then there exist $\rho\in(0,r)$ and $t_\rho\in(0,1)$ such that
$$\textrm{$\gamma^+_0(t)\subset B_\rho$ for $t<t_\rho$,\qquad while \ $\overline{B_\rho}\cap(\{\gamma^+_0(t) , t>t_\rho\}\cup\varphi^+)=\emptyset$.}$$

\begin{figure}[htb]
  \centering
  \def\svgwidth{250pt}
  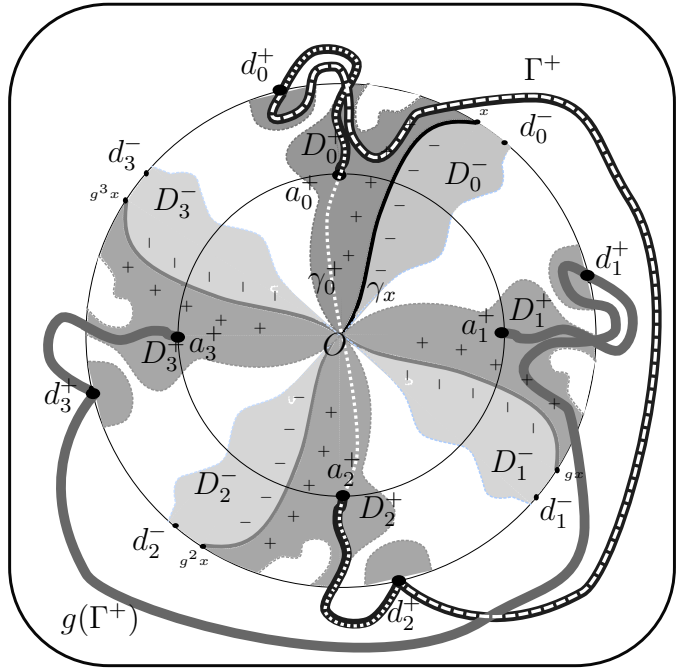
  \caption{$h=4$, $D_0^+=D_2^+$}\label{FigurePrima}
\end{figure}

Then by construction and by symmetry (see Figure \ref{FigurePrima} where $k=0$, $j=2$) the curves $$\Gamma^+:=\{\gamma^+_0(t),t>t_\rho\}\cup\varphi^+\cup\{\gamma^+_j(t),t>t_\rho\}$$ and $g(\Gamma^+)$ are two piecewise regular curves connecting respectively $a_0^+:=\gamma^+_0\cap\partial B_\rho$ with $a_j^+:=\gamma^+_j\cap\partial B_\rho$ in $D^+_0\setminus B_\rho$ and $a_{1}^+:=\gamma^+_{1}\cap\partial B_\rho$ with $a_{j+1}^+:=\gamma^+_{j+1}\cap\partial B_\rho$ in $D^+_{1}\setminus B_\rho$.
Since the points $a_0^+,a_{1}^+,a_j^+,a_{j+1}^+$ are ordered on $\partial B_\rho$ (namely there exists $\theta_0^+\in\R^+$ such that $a^+_i=\rho e^{i\theta_i^+}$, where $\theta^+_i:=\theta^+_0+\frac{2\pi i}{h}$, $i\in\{0,1,j,j+1\}$, and $\theta_0^+<\theta_{1}^+<\theta_j^+<\theta^+_{j+1}<\theta^+_0+2\pi$), it follows that $\Gamma^+ \cap g(\Gamma^+)\neq\emptyset$. Thus $D^+_0=D^+_{1}$ and we are back to the case in which $j=k+1$.\\

So we have proved that $D^+_0=D^+_1=\ldots = D^+_{h-1}$. Next we show that this implies that the sets $D^-_i$, $i=0, \ldots, h-1$ are pairwise disjoint, so that we again reach  a contradiction because $u_p$ would have at least $h+1\geq 5$ nodal domains and this concludes the proof. \\

So it remains to prove that if $D^+_0=D^+_1=\ldots = D^+_{h-1}$ then the sets $D^-_i$, $i=0, \ldots, h-1$, must be pairwise disjoint.

\begin{figure}[htb]
  \centering
  \def\svgwidth{270pt}
  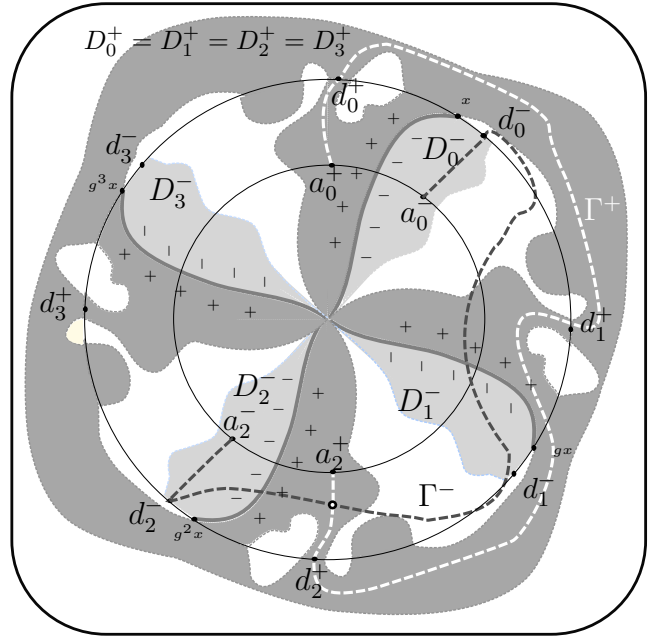
  \caption{$h=4$, $D_0^+=D_1^+=\ldots = D_{h-1}^+$}\label{FigureSeconda}
\end{figure}
Let us suppose by contradiction that $D^-_k=D^-_j$ for some $k<j$ and as before we assume without loss of generality that $k=0$. Since by assumption we also have $D^+_0=D^+_j$ then, similarly as before, we deduce the existence of $\rho\in(0,r)$ such that we can construct (exploiting $\gamma^\pm_0$ and $\gamma^\pm_j$) two piecewise regular curves $\Gamma^\pm\subset(D^\pm_0\setminus B_\rho)$ joining respectively $a^\pm_0:=\gamma^\pm_0\cap B_\rho$ with $a^\pm_j:=\gamma^\pm_j\cap B_\rho$. By definition of $a^\pm_i$ and we derive the existence of $\theta^\pm_0\in\R^+$, $0<|\theta^+_0-\theta^-_0|<\tfrac{2\pi}{h}$, such that $a^\pm_0=\rho e^{i\theta^\pm_0}$ and $a^\pm_j=\rho e^{i\theta^\pm_j}$, where $\theta^\pm_j:=\theta^\pm_0+\tfrac{2\pi j}{h}$. If without loss of generality we assume $\theta_0^+<\theta_0^-$, then we get that $a^+_0,a^-_0,a^+_j,a^-_j$ are ordered on $\partial B_\rho$ in the following sense: $\theta^+_0<\theta^-_0<\theta^+_j<\theta^-_j<\theta^+_0+2\pi$. In turn this implies that $\Gamma^+\cap\Gamma^-\neq\emptyset$ since $u_p<0$ on $\Gamma^+$ while $u_p>0$ on $\Gamma^-$ (see Figure \ref{FigureSeconda}).
\end{proof}

\ble\label{lemmaB}
Let $u_p$ be a $G$-symmetric sign-changing solution to \eqref{problem} with at most four nodal regions, then each nodal region is $G$-symmetric.
\ele

\begin{proof}
By Lemma \ref{lemmaA} we have that $O$ is in the interior of a nodal domain $\Omega_0$.\\
It is easy to show that $\Omega_0$ is $G$-symmetric, indeed if this is not the case there exists $x\in\Omega_0$ and $k\in \{1,\ldots,h-1\}$ such that $g^k(x)\notin\Omega_0$. Since $\Omega_0$ is connected then there exists a piecewise regular curve $\gamma_0\subset\Omega_0$ joining $O$ with $x$. By symmetry the curve $g^k(\gamma_0)$ connects $O$ with $g^k(x)$ and $u_p(g^k(\gamma_0(t))=u_p(\gamma_0(t))\neq0$ for any $t\in[0,1]$. So the union of the two curves joins $x$ with $g^k(x)$ and is contained in $\Omega_0$, which is a contradiction.

\

Let us consider now a nodal domain $D$ of $u_p$ in $\Omega\setminus\Omega_0$ and we assume by contradiction that it is not $G$-symmetric, namely that there exists $j\in\{1,\ldots,h-1\}$ and $x\in D$ such that $g^j(x)\notin D$. Note that without loss of generality we can take $j=1$, otherwise $g(x)\in D$ and by the action of the group $g^j(x)\in D$ for any $j$. We can divide the proof in two possible cases.

\

\emph{Case 1.}
If for any $i\in\{1,\ldots,h-1\}$ $g^i(x)\notin D$, then the sets $g^i(D)$, $i\in\{0,\ldots,h-1\}$, are $h$ disjoint nodal regions in which $u_p$ has the same sign, hence $u_p$ has at least $h+1\geq 5$ nodal regions which is a contradiction against the assumption.

\

\emph{Case 2.}
If there exists $k\in\{2,\ldots,h-1\}$ such that $g^{k}(x)\in D$, then either $g^k$ generates $C_h$ or $g^k$ generates a proper subgroup $C_\ell$ of $C_h$, $0<\ell<h$. In the first case we reach a contradiction since $g(x)=g^{km}(x)\in D$ for a certain $m\in\N$. In the second case we consider a piecewise regular curve $\gamma_0$ in $D$, connecting $x$ with $g^k(x)$ and define a new curve $\gamma:=\gamma_0\cup g^k(\gamma_0)\cup\ldots\cup g^{\ell k}(\gamma_0)$. By symmetry this is a closed curve in $D$, around the origin.

Of course also $g(\gamma)$ is a closed curve around the origin. Then $g(\gamma)\cap\gamma\neq\emptyset$ otherwise $\gamma$ would lie on one side with respect to $g(\gamma)$ and then the diameter of $\gamma$ would be different from the diameter of $g(\gamma)$ which is impossible since $g$ is an isometry in $\R^2$. Hence $g(\gamma)$ and $\gamma$ intersect and this implies that $g(x)$ belongs to $D$ which is a contradiction.
\end{proof}

\ble\label{lemmaC}
Let $u_p$ be a $G$-symmetric sign-changing solution to \eqref{problem} with at most four nodal regions, then its nodal line does not touch the boundary.
\ele

\begin{proof}
Let us suppose by contradiction that there exists a solution $u_{p}$ of \eqref{problem} having the nodal set which touches the boundary, namely $$\mathcal{Z}_{p,0}:=\overline{ \mathcal{Z}_{p}}\cap\partial \Omega=\emptyset.$$

\

To reach the contradiction, we will show, exploiting the symmetry, that if ${u_{p}}$ had a nodal line touching the boundary, then it would have at least $\ns$ nodal regions.

\

We fix a smooth parametrization $\phi:[0,2\pi)\to\partial \Omega$ of the boundary of $\Omega$ and henceforth, given two points $x,y\in\partial \Omega$ (such that $\phi^{-1}(x)<\phi^{-1}(y)$), we will keep the following notation:
$$
\arcup{xy}:=\set{\phi(\theta)\,|\,\phi^{-1}(x)<\theta<\phi^{-1}(y)}.
$$

Let us notice that $\zpo$ cannot contain any open subset of $\partial  \Omega$, namely
\bel\label{zpopti}
\not\exists\;\; y_1,\,y_2\in\partial \Omega\textnormal{  such that  }\arcup{y_1 y_2}\subset\zpo.
\eel

Indeed, if for some $y_1,\,y_2\in\partial \Omega$ $\arcup{y_1 y_2}\subset\zpo$, then we would have that
$$\arcup{y_1 y_2}\subset\partial\{{u_{p}(x)}>0,\,x\in \Omega\}\textnormal{ and }\arcup{y_1 y_2}\subset\partial\{{u_{p}}<0,\,x\in \Omega\},$$ but this is impossible by Hopf Lemma. Anyway, this also follows by results on the nodal line of \cite{HW} and \cite{HHT}.

\

From the latter considerations and from the closeness of $\zpo$ we deduce the existence of $0\leq\theta_m<\theta_M<2\pi$ such that
\bel\label{thetam}
\arcup{{\phi}(\theta_m){\phi}(\theta_M)}\cap\mathcal{Z}_{p,0}=\emptyset\textnormal{ and  }{\phi}(\theta_m),{\phi}(\theta_M)\in\mathcal{Z}_{p,0}.
\eel
It will be convenient to assume, up to a reparametrization, that $\theta_m=0$.

\

Let us consider
\be
a_0:={\phi}(0),\;a_1:=g({\phi}(0)),\;a_2:=g^2({\phi}(0)),\ldots,\;a_{h-1}:=g^{h-1}({\phi}(0)),
\ee
and
\be
b_0:={\phi}(\theta_M),\;b_1:=g({\phi}(\theta_M)),\;b_2:=g^2({\phi}(\theta_M)),\ldots,\;b_{h-1}:=g^{h-1}({\phi}(\theta_M)).
\ee
By the symmetry of $u_{p}$ for any $i\in\set{0,\ldots,h-1}$, $a_i$, $b_i\in\mathcal{Z}_{p,0}$ and $\arcup{a_i b_i}\cap\zpo=\emptyset$.

\

We will denote by $D$ the nodal domain of $u_p$ having $\arcup{a_0b_0}$ in its boundary. By Lemma \ref{lemmaB} $D$ is $G$-symmetric and therefore it has also $\arcup{a_1 b_1},\ldots,\arcup{a_{h-1}b_{h-1}}$ in its boundary.\\
Moreover $b_0$ necessarily belongs to the boundary of two disjoint nodal regions where $u_p$ has different sign, because otherwise, since $b_0\in\zpo$, $b_0$ would be on the boundary of two disjoint nodal regions in which $u_p$ has the same sign, but this is impossible by the strong maximum principle. The same applies to the points $b_1,\ldots,b_{h-1}$.

\begin{figure}[htb]
  \centering
  \def\svgwidth{200pt}
  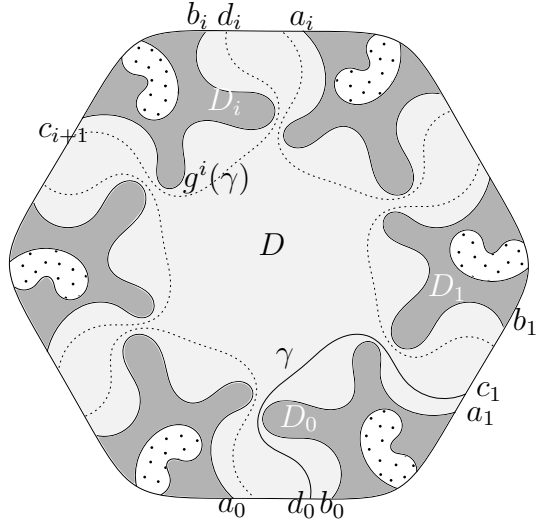
  \caption{The $D_i$'s are pairwise disjoint.}\label{FigureC}
\end{figure}

Thus, if ${u_{p}}|_D>0$ (respectively $<0$), for any $i=0,\ldots,h-1$ there exists an open connected nodal region $D_i$ having $b_i$ in its boundary and in which ${u_{p}}$ is negative (respectively positive). On the one hand by Lemma \ref{lemmaB} all the $D_i$'s coincide, being $G$-symmetric, but on the other hand it is possible to connect in $D$ a point $d_0$ of $\arcup{a_0b_0}$ with a point $c_1$ of $\arcup{a_1b_1}$ by a piecewise regular curve $\gamma$ (and analogously, for any $i\in{1,\ldots,h-1}$, $g^i(d_0)$ with $g^i(c_1)$ by $g(\gamma)$), and thus exploiting the properties of $\R^2$ the $D_i$'s turn out to be pairwise disjoint (see Figure \ref{FigureC}). This contradiction concludes the proof.
\end{proof}

\begin{remark}\label{rmkLemmi}
We point out that the last two lemmas hold true also for a domain $\Omega$ of the following form: $\Omega=\Omega_1\setminus\Omega_0$, where $O\in\Omega_0$ and $\Omega_0$, $\Omega_1$ are $G$-symmetric.
\end{remark}

\begin{proof}[Proof of Theorem \ref{teoremaPrincipale}]

Let us start by defining
\bel\label{u1u2primo}u_1:=u_{p,1,\bar\alpha}\mbox{ and } u_2:=u_{p,2,\bar\alpha},
\eel
where $u_{p,1,\bar\alpha}$, $u_{p,2,\bar\alpha}$  are the functions defined in Corollary \ref{PropositionEnergyUpperBound}.
With this choice of $u_1$ and $u_2$ we consider the set $\mathcal O_p(u_1,u_2)$ defined in \eqref{O_p}.
Then, by the results in the previous section and in particular by Theorem \ref{theoremSolution}
we have that there exist $t_1,t_2\in\mathbb R$, $t_1\cdot t_2 <0$ such that taking the initial condition
$v_{p,0}=t_1u_1+t_2u_2,$  the corresponding solution $v_p(x,t)$ of \eqref{ParabolicProblem} is global in time, sign changing for every fixed time, the $\omega$-limit set $\omega(v_{p,0})$ is nonempty  and any function $u_p\in\omega(v_{p,0})$ is a nodal solution of \eqref{problem}.

\

Now with this special choice of the functions $u_1$ and $u_2$, using the energy estimates proved in Section \ref{SectionEnergyEstimates}, we show that if $p$ is sufficiently large, any function $u_p\in\omega(v_{p,0})$ has at most four nodal domains.
Indeed since the energy is nonincreasing along trajectories,
it satisfies
 \begin{equation}E_p(u_p)\leq E_p(v_{p,0}) .\end{equation}

Thus, thanks to the energy estimate given by Corollary \ref{PropositionEnergyUpperBound} it follows that
for any $\epsilon >0$ there exists $p_{\epsilon}$ such that  for $p\geq p_{\epsilon}$
\begin{equation} \label{stimaEnergiaSoluzioneStep1} pE_p(u_p)\leq  4.97\ \cdot 4\pi e +\ \epsilon .\end{equation}
Finally, \eqref{stimaEnergiaSoluzioneStep1} combined with the energy estimate in Proposition \ref{PropositionEnergyLowerBound} implies that $u_p$ has at most four nodal domains if $p$ is sufficiently large.

\

It is worth to point out that if the initial datum $v_{p,0}$ of \eqref{ParabolicProblem} is $G$-symmetric, then, by the uniqueness of the solution (see Proposition \ref{prop:localex}), the solution $v_p(x,t)$ of \eqref{ParabolicProblem} (for every fixed time) as well as any function in $\omega(v_{p,0})$ turns out to be $G$-symmetric \\
By this remark we immediately deduce that the solution $u_p$ we have found is $G$-symmetric and hence, by Lemma \ref{lemmaA} and Lemma \ref{lemmaC}, its nodal line does not contain the origin and does not touch the boundary.

\

Thus if $u_p$ has two nodal domains the proof is complete.

Otherwise, if $u_p$ has more than two nodal domains we will choose a new $G$-symmetric initial datum - which will be nothing but a restriction of $u_p$ to a subset of $\Omega$ given by two nodal regions - and we will restart the procedure.\\

More precisely let us first observe that in each nodal region $D^i_p$ $i\in\{1,\ldots,k\}$, $k=3$ or $k=4$, the lower energy bound given by Proposition \ref{PropositionEnergyLowerBound} holds, i.e.
\bel\label{lowerbound2}
p E_p(u_p\chi_{D^i_p})\geq 4\pi e-\ep
\eel
for any $\ep>0$, if $p$ is sufficiently large.

By Lemma \ref{lemmaB} we know that each nodal region is $G$-symmetric. We then choose two nodal regions, say $D^1_p$ and $D^2_p$, where $u_p$ has different sign and consider the functions
\bel\label{u1u2secondo}
u_1:= u_p\chi_{D^1_p}\mbox{  and  }u_2:= u_p\chi_{D^2_p}
\eel
which are obviously $G$-symmetric. With this new pair of functions, having obviously disjoint supports, we repeat the same argument applied at the beginning of this proof to the functions defined in \eqref{u1u2primo}. Thus, we obtain the existence of another $G$-symmetric, sign-changing initial condition $v_{p,0}=t_1 u_1+t_2 u_2$, such that the corresponding solution $v_p$ of \eqref{ParabolicProblem}, is global in time, $G$-symmetric and sign changing for every fixed time, the $\omega$-limit set $\omega(v_{p,0})$ is nonempty and any function $\tilde u_p\in\omega(v_{p,0})$ is a nodal solution to \eqref{problem} which satisfies by \eqref{lowerbound2}
\begin{eqnarray}                              \label{stimaEnergiaSoluzioneStep2}
pE_p(\tilde u_p)&\leq&pE_p(t_1u_1+t_2u_2)\nonumber\\
&\leq &pE_p(u_1)+pE_p(u_2)\nonumber\\
&=& pE_p(u_p)-\sum_{i=3}^k pE_p(u_p\chi_{D_p^i})\nonumber\\
&\leq&  4.97\ \cdot 4\pi e +\ \epsilon -\sum_{i=3}^k pE_p(u_p\chi_{D_p^i})\nonumber\\
&\leq& 4.97\ \cdot 4\pi e  -\sum_{i=3}^k 4\pi e +\ \epsilon
\end{eqnarray}
where $k\in\{3,4\}$ is the total number of nodal regions of $u_p$.

\

Then if $k=4$, again by \eqref{lowerbound2}, we get that $\tilde u_p$ has two nodal regions and it is the required solution of \eqref{problem}. If $k=3$ $\tilde u_p$ may have $3$ nodal regions and so, if this is the case, we restart the same procedure by considering the restriction of $\tilde u_p$ to only two nodal regions and get our final assertion.
\end{proof}

\

\

\

We conclude proving Theorem \ref{teoremaMorseIndex}.
\begin{proof}[Proof of Theorem \ref{teoremaMorseIndex}]
In the first part of this proof we will follow closely the paper \cite{AftalionPacella}, where it is proved an analogous result in the case of a ball or an annulus. For sake of clarity we repeat briefly the argument.

Without loss of generality we assume that the symmetry line is the $x_2$-axis
$$T=\set{x=(x_1,x_2)\in\R^2, x_1=0}$$
and we consider the half-domains
$$
\Omega^-=\set{x=(x_1,x_2)\in\Omega, x_1<0}, \qquad \Omega^+=\set{x=(x_1,x_2)\in\Omega, x_1>0}.
$$
We denote by $L$ the linearized operator, $L=-\lap-p|u_p|^{p-1}$ and by $\lambda_k$ the eigenvalues of $L$ in $\Omega$ with homogeneous Dirichlet boundary conditions. Moreover, let $\mu$ be the first eigenvalue of $L$ in $\Omega^-$, namely $\mu$ is such that there exists a solution $\psi$  to
\be
-\lap \psi-p|u_p|^{p-1}\psi=\mu\psi\quad\textnormal{in $\Omega^-$,}\qquad \psi=0\quad\textnormal{on $\partial\Omega^-$,}\qquad \psi>0\quad\textnormal{on $\partial\Omega^-$}.
 \ee
Proposition 2.1 of \cite{AftalionPacella} implies that the odd extension of $\psi$ to $\Omega$, defined by
$$
\tilde{\psi}(x_1,x_2):=\left\{ \begin{array}{lr}\psi(x_1,x_2) & \mbox{ if }(x_1,x_2)\in\Omega^-
\\-\psi(-x_1,x_2) & \mbox{ if }  (x_1,x_2) \in \Omega^+\end{array}\right.
$$
is an eigenfunction for the linearized operator $L$ in $\Omega$ with corresponding eigenvalue $\mu$. Hence $\mu=\lambda_\beta$, with $\beta\geq2$.

\

Next we consider the partial derivative $\tfrac{\partial u_p}{\partial x_1}$ in $\Omega^-$ and we observe that, by the symmetry of $u_p$ in the $x_1$ direction, $\tfrac{\partial u_p}{\partial x_1}=0$ on $\partial \Omega^-\cap T$. Furthermore, since the nodal line of $u_p$ does not touch the boundary, we can assume that $u_p>0$ near $\partial \Omega$ so that the convexity of $\Omega$ in the $x_1$ direction implies that $\tfrac{\partial u_p}{\partial x_1}\geq0$ on $\partial \Omega^-\cap\partial \Omega$. Thus $\tfrac{\partial u_p}{\partial x_1}$ solves
\bel\label{problemEigenvD}
-\lap(\frac{\partial u_p}{\partial x_1})-p|u_p|^{p-1}\frac{\partial u_p}{\partial x_1}=0\quad \textnormal{in $\Omega^-$,}\qquad\frac{\partial u_p}{\partial x_1}\geq0\quad\textnormal{on $\partial\Omega^-$.}
\eel\\
Besides, since $u_p$ is sign-changing, also $\tfrac{\partial u_p}{\partial x_1}$ must change sign and be negative somewhere in $\Omega^-$. Therefore there exists a connected region $D$ strictly contained in $\Omega^-$ such that $\tfrac{\partial u_p}{\partial x_1}<0$ in $D$ and $\tfrac{\partial u_p}{\partial x_1}=0$ on $\partial D$.\\
Hence, the first eigenvalue of $L$ in $D$ is zero and so the first eigenvalue $\mu$ in $\Omega^-$ is negative. Since $\lambda_1<\lambda_\beta=\mu<0$ we have already obtained two negative eigenvalues of $L$. We denote by $\phi_1$ and $\phi_\beta$, respectively, the positive eigenfunction corresponding to the first eigenvalue and the eigenfunction corresponding to $\lambda_\beta$, obtained as the odd extension of the first eigenfunction of $L$ in $\Omega^-$. Clearly $\phi_\beta$ is sign-changing and is not $G$-symmetric.

\

On the other hand if we restrict our attention to the subspace $H^G$ of $H^1_0(\Omega)$ of $G$-invariant functions, we see that $L^G:=L|_{H^G}$ has at least two negative eigenvalues ($\lambda_1^G$, $\lambda_2^G$) being $u_p$ a sign-changing solution in $H^G$. These eigenvalues cannot be equal to $\lambda_\beta$: the first one $\lambda_1^G$ for obvious reasons and $\lambda_2^G$ because its eigenfunction is $G$ symmetric while the eigenfunction corresponding to $\lambda_\beta$ does not changes sign in $\Omega^-$ and is odd with respect to $T$. This proves that $m(u_p)\geq3$.
\end{proof}

\begin{remark}
Note that if $\Omega$ is symmetric with respect to two orthogonal lines $T_1$, $T_2$ in $\R^2$ then $m(u_p)\geq 4$ since the same proof of Theorem \ref{teoremaMorseIndex} yields the existence of two negative eigenvalues of $L$, whose corresponding eigenfunctions are odd with respect to $T_1$ or $T_2$ and have only two nodal regions.
\end{remark}

\end{document}

%% file: DisegnoNumero1orizzo.pdf_tex
\begingroup%
  \makeatletter%
  \providecommand\color[2][]{%
    \errmessage{(Inkscape) Color is used for the text in Inkscape, but the package 'color.sty' is not loaded}%
    \renewcommand\color[2][]{}%
  }%
  \providecommand\transparent[1]{%
    \errmessage{(Inkscape) Transparency is used (non-zero) for the text in Inkscape, but the package 'transparent.sty' is not loaded}%
    \renewcommand\transparent[1]{}%
  }%
  \providecommand\rotatebox[2]{#2}%
  \ifx\svgwidth\undefined%
    \setlength{\unitlength}{488.95bp}%
    \ifx\svgscale\undefined%
      \relax%
    \else%
      \setlength{\unitlength}{\unitlength * \real{\svgscale}}%
    \fi%
  \else%
    \setlength{\unitlength}{\svgwidth}%
  \fi%
  \global\let\svgwidth\undefined%
  \global\let\svgscale\undefined%
  \makeatother%
  \begin{picture}(1,1.00260763)%
    \put(0,0){\includegraphics[width=\unitlength]{DisegnoNumero1orizzo.pdf}}%
    \put(0.70468871,0.83329629){\color[rgb]{0,0,0}\makebox(0,0)[lb]{\smash{\sm x}}}%
    \put(0.62764926,0.78197226){\color[rgb]{0,0,0}\makebox(0,0)[lb]{\smash{\m}}}%
    \put(0.59674812,0.74395964){\color[rgb]{0,0,0}\makebox(0,0)[lb]{\smash{\m}}}%
    \put(0.57567905,0.6918683){\color[rgb]{0,0,0}\makebox(0,0)[lb]{\smash{\m}}}%
    \put(0.5616329,0.64127723){\color[rgb]{0,0,0}\makebox(0,0)[lb]{\smash{\m}}}%
    \put(0.65232797,0.72957777){\color[rgb]{0,0,0}\makebox(0,0)[lb]{\smash{$D_0^-$}}}%
    \put(0.4370631,0.77321261){\color[rgb]{0,0,0}\makebox(0,0)[lb]{\smash{$D_0^+$}}}%
    \put(0.78147347,0.37205966){\color[rgb]{0,0,0}\rotatebox{-90}{\makebox(0,0)[lb]{\smash{\m}}}}%
    \put(0.74346085,0.4029607){\color[rgb]{0,0,0}\rotatebox{-90}{\makebox(0,0)[lb]{\smash{\m}}}}%
    \put(0.69136951,0.42402987){\color[rgb]{0,0,0}\rotatebox{-90}{\makebox(0,0)[lb]{\smash{\m}}}}%
    \put(0.64077858,0.43807581){\color[rgb]{0,0,0}\rotatebox{-90}{\makebox(0,0)[lb]{\smash{\m}}}}%
    \put(0.59281808,0.45212203){\color[rgb]{0,0,0}\rotatebox{-90}{\makebox(0,0)[lb]{\smash{\m}}}}%
    \put(0.6156794,0.49990984){\color[rgb]{0,0,0}\rotatebox{-90}{\makebox(0,0)[lb]{\smash{\p}}}}%
    \put(0.67375711,0.48856701){\color[rgb]{0,0,0}\rotatebox{-90}{\makebox(0,0)[lb]{\smash{\p}}}}%
    \put(0.725131,0.46673976){\color[rgb]{0,0,0}\rotatebox{-90}{\makebox(0,0)[lb]{\smash{\p}}}}%
    \put(0.77237853,0.44315995){\color[rgb]{0,0,0}\rotatebox{-90}{\makebox(0,0)[lb]{\smash{\p}}}}%
    \put(0.81302161,0.40840449){\color[rgb]{0,0,0}\rotatebox{-90}{\makebox(0,0)[lb]{\smash{\p}}}}%
    \put(0.36651727,0.22301816){\color[rgb]{0,0,0}\rotatebox{180}{\makebox(0,0)[lb]{\smash{\m}}}}%
    \put(0.39741845,0.26103088){\color[rgb]{0,0,0}\rotatebox{180}{\makebox(0,0)[lb]{\smash{\m}}}}%
    \put(0.41848753,0.31312222){\color[rgb]{0,0,0}\rotatebox{180}{\makebox(0,0)[lb]{\smash{\m}}}}%
    \put(0.4325337,0.36371318){\color[rgb]{0,0,0}\rotatebox{180}{\makebox(0,0)[lb]{\smash{\m}}}}%
    \put(0.44657969,0.41167345){\color[rgb]{0,0,0}\rotatebox{180}{\makebox(0,0)[lb]{\smash{\m}}}}%
    \put(0.49436745,0.38881232){\color[rgb]{0,0,0}\rotatebox{180}{\makebox(0,0)[lb]{\smash{\p}}}}%
    \put(0.48302457,0.33073441){\color[rgb]{0,0,0}\rotatebox{180}{\makebox(0,0)[lb]{\smash{\p}}}}%
    \put(0.46119737,0.27936052){\color[rgb]{0,0,0}\rotatebox{180}{\makebox(0,0)[lb]{\smash{\p}}}}%
    \put(0.43761766,0.23211299){\color[rgb]{0,0,0}\rotatebox{180}{\makebox(0,0)[lb]{\smash{\p}}}}%
    \put(0.4028622,0.19146991){\color[rgb]{0,0,0}\rotatebox{180}{\makebox(0,0)[lb]{\smash{\p}}}}%
    \put(0.21641321,0.62947238){\color[rgb]{0,0,0}\rotatebox{90}{\makebox(0,0)[lb]{\smash{\m}}}}%
    \put(0.25442583,0.59857115){\color[rgb]{0,0,0}\rotatebox{90}{\makebox(0,0)[lb]{\smash{\m}}}}%
    \put(0.30651732,0.57750208){\color[rgb]{0,0,0}\rotatebox{90}{\makebox(0,0)[lb]{\smash{\m}}}}%
    \put(0.35710834,0.56345602){\color[rgb]{0,0,0}\rotatebox{90}{\makebox(0,0)[lb]{\smash{\m}}}}%
    \put(0.4050686,0.54940991){\color[rgb]{0,0,0}\rotatebox{90}{\makebox(0,0)[lb]{\smash{\m}}}}%
    \put(0.38220738,0.5016222){\color[rgb]{0,0,0}\rotatebox{90}{\makebox(0,0)[lb]{\smash{\p}}}}%
    \put(0.32412952,0.51296504){\color[rgb]{0,0,0}\rotatebox{90}{\makebox(0,0)[lb]{\smash{\p}}}}%
    \put(0.27275573,0.53479228){\color[rgb]{0,0,0}\rotatebox{90}{\makebox(0,0)[lb]{\smash{\p}}}}%
    \put(0.22550815,0.55837199){\color[rgb]{0,0,0}\rotatebox{90}{\makebox(0,0)[lb]{\smash{\p}}}}%
    \put(0.18486507,0.59312745){\color[rgb]{0,0,0}\rotatebox{90}{\makebox(0,0)[lb]{\smash{\p}}}}%
    \put(0.74885437,0.52440378){\color[rgb]{0,0,0}\makebox(0,0)[lb]{\smash{$D_1^+$}}}%
    \put(0.72157473,0.29392854){\color[rgb]{0,0,0}\makebox(0,0)[lb]{\smash{$D_1^-$}}}%
    \put(0.52406314,0.2246614){\color[rgb]{0,0,0}\makebox(0,0)[lb]{\smash{$D_2^+$}}}%
    \put(0.27725996,0.26090516){\color[rgb]{0,0,0}\makebox(0,0)[lb]{\smash{$D_2^-$}}}%
    \put(0.19770358,0.45904539){\color[rgb]{0,0,0}\makebox(0,0)[lb]{\smash{$D_3^+$}}}%
    \put(0.21676916,0.6901259){\color[rgb]{0,0,0}\makebox(0,0)[lb]{\smash{$D_3^-$}}}%
    \put(0.25542607,0.15447253){\color[rgb]{0,0,0}\makebox(0,0)[lb]{\smash{\sm {g^2x}}}}%
    \put(0.12135195,0.70781238){\color[rgb]{0,0,0}\makebox(0,0)[lb]{\smash{\sm{g^3x}}}}%
    \put(0.77239663,0.8865367){\color[rgb]{0,0,0}\makebox(0,0)[lb]{\smash{$\Gamma^+$}}}%
    \put(0.07891707,0.05914959){\color[rgb]{0,0,0}\makebox(0,0)[lb]{\smash{$g(\Gamma^+)$}}}%
    \put(0.49979908,0.61617819){\color[rgb]{0,0,0}\makebox(0,0)[lb]{\smash{\p}}}%
    \put(0.51114201,0.674256){\color[rgb]{0,0,0}\makebox(0,0)[lb]{\smash{\p}}}%
    \put(0.53296916,0.72562989){\color[rgb]{0,0,0}\makebox(0,0)[lb]{\smash{\p}}}%
    \put(0.55654887,0.77287742){\color[rgb]{0,0,0}\makebox(0,0)[lb]{\smash{\p}}}%
    \put(0.59130433,0.8135205){\color[rgb]{0,0,0}\makebox(0,0)[lb]{\smash{\p}}}%
    \put(0.34812451,0.89554616){\color[rgb]{0,0,0}\makebox(0,0)[lb]{\smash{$d_0^+$}}}%
    \put(0.76570189,0.80206489){\color[rgb]{0,0,0}\makebox(0,0)[lb]{\smash{$d_0^-$}}}%
    \put(0.87917269,0.60305546){\color[rgb]{0,0,0}\makebox(0,0)[lb]{\smash{$d_1^+$}}}%
    \put(0.79568123,0.22241079){\color[rgb]{0,0,0}\makebox(0,0)[lb]{\smash{$d_1^-$}}}%
    \put(0.56682937,0.06810353){\color[rgb]{0,0,0}\makebox(0,0)[lb]{\smash{$d_2^+$}}}%
    \put(0.18550942,0.17629653){\color[rgb]{0,0,0}\makebox(0,0)[lb]{\smash{$d_2^-$}}}%
    \put(0.05349638,0.39733341){\color[rgb]{0,0,0}\makebox(0,0)[lb]{\smash{$d_3^+$}}}%
    \put(0.14915032,0.76212649){\color[rgb]{0,0,0}\makebox(0,0)[lb]{\smash{$d_3^-$}}}%
    \put(0.54758689,0.59331697){\color[rgb]{0,0,0}\makebox(0,0)[lb]{\smash{\m}}}%
    \put(0.47080421,0.47103782){\color[rgb]{0,0,0}\rotatebox{0.03046503}{\makebox(0,0)[lb]{\smash{$O$}}}}%
    \put(0.83229319,0.28589893){\color[rgb]{0,0,0}\makebox(0,0)[lb]{\smash{\sm {gx}}}}%
    \put(0.67828432,0.50199537){\color[rgb]{0,0,0}\makebox(0,0)[lb]{\smash{$a_1^+$}}}%
    \put(0.53704184,0.55939845){\color[rgb]{0,0,0}\makebox(0,0)[lb]{\smash{$\gamma_x$}}}%
    \put(0.44868925,0.57412388){\color[rgb]{0,0,0}\makebox(0,0)[lb]{\smash{$\gamma_0^+$}}}%
    \put(0.41432991,0.70338045){\color[rgb]{0,0,0}\makebox(0,0)[lb]{\smash{$a_0^+$}}}%
    \put(0.2703479,0.47595433){\color[rgb]{0,0,0}\makebox(0,0)[lb]{\smash{$a_3^+$}}}%
    \put(0.47699056,0.28584101){\color[rgb]{0,0,0}\makebox(0,0)[lb]{\smash{$a_2^+$}}}%
  \end{picture}%
\endgroup%

%% file: lemmaANegativeBiancoNeroOrizzo.pdf_tex
\begingroup%
  \makeatletter%
  \providecommand\color[2][]{%
    \errmessage{(Inkscape) Color is used for the text in Inkscape, but the package 'color.sty' is not loaded}%
    \renewcommand\color[2][]{}%
  }%
  \providecommand\transparent[1]{%
    \errmessage{(Inkscape) Transparency is used (non-zero) for the text in Inkscape, but the package 'transparent.sty' is not loaded}%
    \renewcommand\transparent[1]{}%
  }%
  \providecommand\rotatebox[2]{#2}%
  \ifx\svgwidth\undefined%
    \setlength{\unitlength}{553.95bp}%
    \ifx\svgscale\undefined%
      \relax%
    \else%
      \setlength{\unitlength}{\unitlength * \real{\svgscale}}%
    \fi%
  \else%
    \setlength{\unitlength}{\svgwidth}%
  \fi%
  \global\let\svgwidth\undefined%
  \global\let\svgscale\undefined%
  \makeatother%
  \begin{picture}(1,0.9291001)%
    \put(0,0){\includegraphics[width=\unitlength]{lemmaANegativeBiancoNeroOrizzo.pdf}}%
    \put(0.44418556,0.74207829){\color[rgb]{0,0,0}\makebox(0,0)[lb]{\smash{$d_0^+$}}}%
    \put(0.67585,0.70797383){\color[rgb]{0,0,0}\makebox(0,0)[lb]{\smash{$d_0^-$}}}%
    \put(0.78954536,0.42400301){\color[rgb]{0,0,0}\makebox(0,0)[lb]{\smash{\small{$d_1^+$}}}}%
    \put(0.71031634,0.19233354){\color[rgb]{0,0,0}\makebox(0,0)[lb]{\smash{\small{$d_1^-$}}}}%
    \put(0.39440331,0.06396251){\color[rgb]{0,0,0}\makebox(0,0)[lb]{\smash{$d_2^+$}}}%
    \put(0.16373667,0.15563256){\color[rgb]{0,0,0}\makebox(0,0)[lb]{\smash{$d_2^-$}}}%
    \put(0.04038021,0.45208418){\color[rgb]{0,0,0}\makebox(0,0)[lb]{\smash{\small{$d_3^+$}}}}%
    \put(0.13164396,0.67272177){\color[rgb]{0,0,0}\makebox(0,0)[lb]{\smash{\small{$d_3^-$}}}}%
    \put(0.10710742,0.62478087){\color[rgb]{0,0,0}\makebox(0,0)[lb]{\smash{\sm{g^3x}}}}%
    \put(0.62466409,0.74243118){\color[rgb]{0,0,0}\makebox(0,0)[lb]{\smash{\sm x}}}%
    \put(0.55399625,0.69023882){\color[rgb]{0,0,0}\makebox(0,0)[lb]{\smash{\m}}}%
    \put(0.52672094,0.65668657){\color[rgb]{0,0,0}\makebox(0,0)[lb]{\smash{\m}}}%
    \put(0.50812409,0.61070763){\color[rgb]{0,0,0}\makebox(0,0)[lb]{\smash{\m}}}%
    \put(0.49572616,0.56605288){\color[rgb]{0,0,0}\makebox(0,0)[lb]{\smash{\m}}}%
    \put(0.48332833,0.52372022){\color[rgb]{0,0,0}\makebox(0,0)[lb]{\smash{\m}}}%
    \put(0.56986446,0.66656182){\color[rgb]{0,0,0}\makebox(0,0)[lb]{\smash{$D_0^-$}}}%
    \put(0.68977089,0.32842508){\color[rgb]{0,0,0}\rotatebox{-90}{\makebox(0,0)[lb]{\smash{\m}}}}%
    \put(0.65621863,0.35570022){\color[rgb]{0,0,0}\rotatebox{-90}{\makebox(0,0)[lb]{\smash{\m}}}}%
    \put(0.61023964,0.37429716){\color[rgb]{0,0,0}\rotatebox{-90}{\makebox(0,0)[lb]{\smash{\m}}}}%
    \put(0.56558495,0.38669491){\color[rgb]{0,0,0}\rotatebox{-90}{\makebox(0,0)[lb]{\smash{\m}}}}%
    \put(0.52325215,0.39909301){\color[rgb]{0,0,0}\rotatebox{-90}{\makebox(0,0)[lb]{\smash{\m}}}}%
    \put(0.54343095,0.44127344){\color[rgb]{0,0,0}\rotatebox{-90}{\makebox(0,0)[lb]{\smash{\p}}}}%
    \put(0.59469387,0.43126157){\color[rgb]{0,0,0}\rotatebox{-90}{\makebox(0,0)[lb]{\smash{\p}}}}%
    \put(0.64003959,0.41199551){\color[rgb]{0,0,0}\rotatebox{-90}{\makebox(0,0)[lb]{\smash{\p}}}}%
    \put(0.68174314,0.39118254){\color[rgb]{0,0,0}\rotatebox{-90}{\makebox(0,0)[lb]{\smash{\p}}}}%
    \put(0.7176172,0.36050525){\color[rgb]{0,0,0}\rotatebox{-90}{\makebox(0,0)[lb]{\smash{\p}}}}%
    \put(0.32350525,0.19687194){\color[rgb]{0,0,0}\rotatebox{180}{\makebox(0,0)[lb]{\smash{\m}}}}%
    \put(0.35078052,0.23042428){\color[rgb]{0,0,0}\rotatebox{180}{\makebox(0,0)[lb]{\smash{\m}}}}%
    \put(0.36937737,0.27640327){\color[rgb]{0,0,0}\rotatebox{180}{\makebox(0,0)[lb]{\smash{\m}}}}%
    \put(0.38177535,0.32105789){\color[rgb]{0,0,0}\rotatebox{180}{\makebox(0,0)[lb]{\smash{\m}}}}%
    \put(0.39417322,0.36339059){\color[rgb]{0,0,0}\rotatebox{180}{\makebox(0,0)[lb]{\smash{\m}}}}%
    \put(0.43635361,0.34321197){\color[rgb]{0,0,0}\rotatebox{180}{\makebox(0,0)[lb]{\smash{\p}}}}%
    \put(0.42634169,0.29194886){\color[rgb]{0,0,0}\rotatebox{180}{\makebox(0,0)[lb]{\smash{\p}}}}%
    \put(0.40707567,0.24660315){\color[rgb]{0,0,0}\rotatebox{180}{\makebox(0,0)[lb]{\smash{\p}}}}%
    \put(0.38626279,0.2048996){\color[rgb]{0,0,0}\rotatebox{180}{\makebox(0,0)[lb]{\smash{\p}}}}%
    \put(0.3555855,0.16902554){\color[rgb]{0,0,0}\rotatebox{180}{\makebox(0,0)[lb]{\smash{\p}}}}%
    \put(0.1910143,0.55563323){\color[rgb]{0,0,0}\rotatebox{90}{\makebox(0,0)[lb]{\smash{\m}}}}%
    \put(0.22456655,0.52835792){\color[rgb]{0,0,0}\rotatebox{90}{\makebox(0,0)[lb]{\smash{\m}}}}%
    \put(0.27054563,0.50976107){\color[rgb]{0,0,0}\rotatebox{90}{\makebox(0,0)[lb]{\smash{\m}}}}%
    \put(0.31520037,0.49736312){\color[rgb]{0,0,0}\rotatebox{90}{\makebox(0,0)[lb]{\smash{\m}}}}%
    \put(0.35753303,0.48496522){\color[rgb]{0,0,0}\rotatebox{90}{\makebox(0,0)[lb]{\smash{\m}}}}%
    \put(0.33735433,0.44278487){\color[rgb]{0,0,0}\rotatebox{90}{\makebox(0,0)[lb]{\smash{\p}}}}%
    \put(0.28609122,0.45279675){\color[rgb]{0,0,0}\rotatebox{90}{\makebox(0,0)[lb]{\smash{\p}}}}%
    \put(0.24074559,0.47206281){\color[rgb]{0,0,0}\rotatebox{90}{\makebox(0,0)[lb]{\smash{\p}}}}%
    \put(0.19904205,0.49287569){\color[rgb]{0,0,0}\rotatebox{90}{\makebox(0,0)[lb]{\smash{\p}}}}%
    \put(0.16316799,0.52355298){\color[rgb]{0,0,0}\rotatebox{90}{\makebox(0,0)[lb]{\smash{\p}}}}%
    \put(0.30761785,0.3343136){\color[rgb]{0,0,0}\makebox(0,0)[lb]{\smash{$D_2^-$}}}%
    \put(0.19132846,0.60916971){\color[rgb]{0,0,0}\makebox(0,0)[lb]{\smash{$D_3^-$}}}%
    \put(0.75197111,0.25370845){\color[rgb]{0,0,0}\makebox(0,0)[lb]{\smash{\sm {gx}}}}%
    \put(0.2254494,0.13636943){\color[rgb]{0,0,0}\makebox(0,0)[lb]{\smash{\sm {g^2x}}}}%
    \put(0.47042583,0.64050766){\color[rgb]{0,0,0}\makebox(0,0)[lb]{\smash{\p}}}%
    \put(0.45115986,0.59516194){\color[rgb]{0,0,0}\makebox(0,0)[lb]{\smash{\p}}}%
    \put(0.49123871,0.68221112){\color[rgb]{0,0,0}\makebox(0,0)[lb]{\smash{\p}}}%
    \put(0.521916,0.71808522){\color[rgb]{0,0,0}\makebox(0,0)[lb]{\smash{\p}}}%
    \put(0.44114789,0.54389893){\color[rgb]{0,0,0}\makebox(0,0)[lb]{\smash{\p}}}%
    \put(0.54016846,0.57829098){\color[rgb]{0,0,0}\makebox(0,0)[lb]{\smash{$a_0^-$}}}%
    \put(0.29449873,0.28779444){\color[rgb]{0,0,0}\makebox(0,0)[lb]{\smash{$a_2^-$}}}%
    \put(0.10177916,0.81955198){\color[rgb]{0,0,0}\makebox(0,0)[lb]{\smash{\footnotesize{$D_0^+=D_1^+=D_2^+=D_3^+$}}}}%
    \put(0.41387369,0.61950026){\color[rgb]{0,0,0}\makebox(0,0)[lb]{\smash{$a_0^+$}}}%
    \put(0.42274574,0.24188578){\color[rgb]{0,0,0}\makebox(0,0)[lb]{\smash{$a_2^+$}}}%
    \put(0.53564722,0.31419371){\color[rgb]{0,0,0}\makebox(0,0)[lb]{\smash{$D_1^-$}}}%
    \put(0.79506273,0.58062145){\color[rgb]{1,1,1}\makebox(0,0)[lb]{\smash{$\Gamma^+$}}}%
    \put(0.56543912,0.18058535){\color[rgb]{0,0,0}\makebox(0,0)[lb]{\smash{$\Gamma^-$}}}%
  \end{picture}%
\endgroup%

%% file: LemmaC1bn.pdf_tex
\begingroup%
  \makeatletter%
  \providecommand\color[2][]{%
    \errmessage{(Inkscape) Color is used for the text in Inkscape, but the package 'color.sty' is not loaded}%
    \renewcommand\color[2][]{}%
  }%
  \providecommand\transparent[1]{%
    \errmessage{(Inkscape) Transparency is used (non-zero) for the text in Inkscape, but the package 'transparent.sty' is not loaded}%
    \renewcommand\transparent[1]{}%
  }%
  \providecommand\rotatebox[2]{#2}%
  \ifx\svgwidth\undefined%
    \setlength{\unitlength}{461.9610089bp}%
    \ifx\svgscale\undefined%
      \relax%
    \else%
      \setlength{\unitlength}{\unitlength * \real{\svgscale}}%
    \fi%
  \else%
    \setlength{\unitlength}{\svgwidth}%
  \fi%
  \global\let\svgwidth\undefined%
  \global\let\svgscale\undefined%
  \makeatother%
  \begin{picture}(1,0.94283541)%
    \put(0,0){\includegraphics[width=\unitlength]{LemmaC1bn.pdf}}%
    \put(0.46699117,0.49075097){\color[rgb]{0,0,0}\makebox(0,0)[lb]{\smash{$D$}}}%
    \put(0.39079426,0.00586156){\color[rgb]{0,0,0}\makebox(0,0)[lb]{\smash{$a_0$}}}%
    \put(0.85663444,0.17903635){\color[rgb]{0,0,0}\makebox(0,0)[lb]{\smash{$a_1$}}}%
    \put(0.58128653,0.00412981){\color[rgb]{0,0,0}\makebox(0,0)[lb]{\smash{$b_0$}}}%
    \put(0.52067535,0.00412981){\color[rgb]{0,0,0}\makebox(0,0)[lb]{\smash{$d_0$}}}%
    \put(0.87459456,0.22493105){\color[rgb]{0,0,0}\makebox(0,0)[lb]{\smash{$c_1$}}}%
    \put(0.94314658,0.34987702){\color[rgb]{0,0,0}\makebox(0,0)[lb]{\smash{$b_1$}}}%
    \put(0.4996685,0.28964234){\color[rgb]{0,0,0}\makebox(0,0)[lb]{\smash{$\gamma$}}}%
    \put(0.50644493,0.167667){\color[rgb]{1,1,1}\makebox(0,0)[lb]{\smash{$D_0$}}}%
    \put(0.78427747,0.4131235){\color[rgb]{1,1,1}\makebox(0,0)[lb]{\smash{$D_1$}}}%
    \put(0.37091679,0.7647436){\color[rgb]{1,1,1}\makebox(0,0)[lb]{\smash{$D_i$}}}%
    \put(0.32574075,0.62018031){\color[rgb]{0,0,0}\makebox(0,0)[lb]{\smash{$g^i(\gamma)$}}}%
    \put(0.52508268,0.92508426){\color[rgb]{0,0,0}\makebox(0,0)[lb]{\smash{$a_i$}}}%
    \put(0.3323316,0.92659013){\color[rgb]{0,0,0}\makebox(0,0)[lb]{\smash{$b_i$}}}%
    \put(0.38955459,0.92583719){\color[rgb]{0,0,0}\makebox(0,0)[lb]{\smash{$d_i$}}}%
    \put(0.054499,0.71350983){\color[rgb]{0,0,0}\makebox(0,0)[lb]{\smash{$c_{i+1}$}}}%
  \end{picture}%
\endgroup%

%% file: DeMarchisIanniPacella.bbl
\begin{thebibliography}{99}


\bibitem{AdimurthiGrossi} Adimurthi, M. Grossi, \emph{Asymptotic estimates for a two-dimensional problem with polynomial nonlinearity}, Proc. Amer. Math. Soc. \textbf{132} (2003), 1013-1019.

\bibitem{Artin} M. Artin, \emph{Algebra}, Prentice Hall, Inc., Englewood Cliffs, NJ, 1991.

\bibitem{AftalionPacella} A. Aftalion, F. Pacella, \emph{Qualitative properties of nodal solutions of semilinear elliptic equations in radially symmetric domains},
    Comptes Rendus Mathematique \textbf{339}  (2004),  339-344.

\bibitem{BartschWethWillem} T. Bartsch, T. Weth, M. Willem, \emph{Partial symmetry of least energy nodal solutions to some variational problems},
 J. Anal. Math. \textbf{96} (2005), 1-18.

\bibitem{EMP} P. Esposito, M. Musso, A. Pistoia, \emph{On the existence and profile of nodal solutions for a two-dimensional elliptic problem with large exponent in nonlinearity}, Proc. Lond. Math. Soc. (3) \textbf{94} (2007),  497-519.

\bibitem{GazzolaWeth} F. Gazzola, T. Weth, \emph{Finite time blow-up and global solutions for semilinear parabolic equations with initial data at high energy level},  Differential and Integral Equations  \textbf{9} (2005), 961-990

\bibitem{Grossi} M. Grossi, \emph{Asymptotic behaviour of the Kazdan-Warner solution in the annulus}, J. Differential Equations \textbf{223} (2006), 96-111.

\bibitem{GrossiGrumiauPacella} M. Grossi, C. Grumiau, F. Pacella, \emph{Lane Emden problems with large exponents and singular Liouville equations} (2012), arXiv:1209.1534.

\bibitem{HW} P. Hartman, A. Wintner, \emph{On the local behavior of solutions of non-parabolic
partial diﬀerential equations}, Amer. J. Math. \textbf{75} (1953), 449–476.

\bibitem{HHT} B. Helffer, T. Hoffmann-Ostenhof, S. Terracini, \emph{Nodal domains and spectral minimal partitions}, Ann. Inst. H. Poincaré Anal. Non Linéaire \textbf{26} (2009), 101-138.

\bibitem{Ianni} I. Ianni, \emph{Sign-changing radial solutions for the Schr\"odinger-Poisson-Slater problem}, TMNA to appear.

\bibitem{NN} W.M. Ni, R.D. Nussbaum, \emph{Uniqueness and nonuniqueness for positive radial solutions of $Au+f(u,r)=0$}, Comm. Pure Appl. Math. \textbf{38} (1985), 67-108.

\bibitem{PacellaWeth} F. Pacella, T. Weth, \emph{Symmetry of solutions to semilinear elliptic equations via Morse index},
Proc. Amer. Math. Soc. \textbf{135} (2007), 1753-1762.

\bibitem{QuittnerSoupletBook}
P. Quittner and P. Souplet, \emph{ Superlinear parabolic problems. Blow-up, global existence and steady states}, Birkh\"auser Verlag 2007.


\bibitem{RenWei} X. Ren, J. Wei, \emph{On a two dimensional elliptic problem with large exponent in nonlinearity}, Trans. Amer. Math. Soc. \textbf{343} (1994), 749-763.

\bibitem{WeiWeth}
J. Wei, T. Weth \emph{Radial solutions and phase separation in a system of two coupled Schr\"odinger equations}, Arch. Rational Mech. Anal. \textbf{190} (2008), 83-106.



\end{thebibliography}
